\documentclass[11pt,leqno]{amsart}
\usepackage{amssymb, amsmath, amsthm, bbm, cite, cases, paralist}
\usepackage[margin=1in]{geometry}
\DeclareMathOperator*{\osc}{osc}
\def\Xint#1{\mathchoice
{\XXint\displaystyle\textstyle{#1}}%
{\XXint\textstyle\scriptstyle{#1}}%
{\XXint\scriptstyle\scriptscriptstyle{#1}}%
{\XXint\scriptscriptstyle\scriptscriptstyle{#1}}%
\!\int}
\def\XXint#1#2#3{{\setbox0=\hbox{$#1{#2#3}{\int}$ }
\vcenter{\hbox{$#2#3$ }}\kern-.6\wd0}}

\def\dashint{\Xint-}

\numberwithin{equation}{section}

\theoremstyle{plain}
\newtheorem{theorem}{Theorem}[section]
\newtheorem{lemma}[theorem]{Lemma}
\newtheorem{proposition}[theorem]{Proposition}
\theoremstyle{remark}
\newtheorem{remark}[theorem]{Remark}
\author{Yan Zhang}
\title{Asymptotic Behavior of a Nonlocal KPP Equation with a Stationary Ergodic Nonlinearity}
\date{November 27, 2014}
\begin{document}
\begin{abstract}
We consider a space-inhomogeneous Kolmogorov-Petrovskii-Piskunov (KPP) equation with a nonlocal diffusion and a stationary ergodic nonlinearity. By employing and adapting the theory of stochastic homogenization, we show that solutions of this equation asymptotically converge to its stationary states in regions of space separated by a front that is determined by a Hamilton-Jacobi variational inequality.
\end{abstract}
\maketitle
\section{Introduction}
\label{sec:intro}
The aim of this paper is to analyze the large space/long time asymptotic behavior of the nonlocal reaction-diffusion equation
\begin{equation} \label{originalequation}
u_t(x, t) - \int J(y)[u(x-y, t) - u(x, t)]dy  - f(x, u) = 0,
\end{equation}
where $J$ is a continuous, compactly supported, and symmetric kernel, and $f$ is a monostable/KPP type nonlinearity in $u$ for which the canonical example is $f(u) = u(1-u)$. To study the asymptotic behavior of \eqref{originalequation}, we introduce the ``hyperbolic'' scaling $(x, t) \mapsto (\epsilon^{-1} x, \epsilon^{-1}t).$ As $\epsilon \rightarrow 0$, the time scaling reproduces long-time behavior of (\ref{originalequation}), while the space scaling reproduces in bounded sets behavior for large space variables. The new unknown is now given by $u^\epsilon(x, t) := u(\epsilon^{-1}x, \epsilon^{-1}t).$ We introduce an initial condition $u^\epsilon(\cdot, 0) = u_0(\cdot)$, and we can easily see that $u^\epsilon$ satisfies
\begin{equation} \label{uequation}
\left\{ \begin{array}{l}
\displaystyle u^\epsilon_t(x, t) - \frac{1}{\epsilon}\int J(y) [u^\epsilon(x-\epsilon y, t) dy - u^\epsilon(x, t)] dy - \frac{1}{\epsilon}f\left(\frac{x}{\epsilon}, u^\epsilon\right)=0  \text{ in } \mathbb{R}^n \times (0, \infty), \\
u^\epsilon(x, 0) = u_0(x).
\end{array}\right.
\end{equation}
Our notion of asymptotic behavior of \eqref{originalequation} is embodied by the behavior of the initial value problem \eqref{uequation} as $\epsilon \rightarrow 0$. In the main result of this work, Theorem \ref{maintheorem}, we prove that as $\epsilon \rightarrow 0$, the $u^\epsilon$ converge respectively to the two equilibria of $f$, which for simplicity we take to be constant, in the two regions $\{\phi < 0\}$ and $\mathrm{int}(\{\phi = 0\})$. $\phi$ is the solution of the Hamilton-Jacobi variational inequality
\begin{equation} \label{effectiveequation}
\left\{
\begin{array}{ll}
\displaystyle \max(\phi_t + \overline{H}(D\phi), \phi) = 0 & \text{ in } \mathbb{R}^n \times (0, \infty) \\
\displaystyle \phi = \left\{
\begin{array}{ll}
0 & \text{ on } G_0 \times \{0\} \\
-\infty & \text{ on } \mathbb{R}^n \backslash \overline{G_0} \times \{0\}.
\end{array} \right.
\end{array} \right.
\end{equation}
$G_0$ is the support of $u_0$, and $\overline{H}(p)$ is an ``effective Hamiltonian'' resulting from the homogenization of \eqref{uequation}. To obtain such a result, it is necessary to make assumptions about the oscillatory behavior of $f$ in the $\frac{x}{\epsilon}$ variable. In this paper we will consider the situation when $f$ is a stationary ergodic process; the situation where $f$ is an almost periodic function is considered by the author in \cite{almostperiodicpaper}. This behavior was shown for a nonlocal equation very close to \eqref{originalequation} that models the propagation of an invasive species in ecology by Perthame and Souganidis in \cite{perthamesouganidis}, and similar asymptotic behavior was found for a non-local Lotka-Volterra equation by Barles, Mirrahimi, and Perthame in \cite{barlesnonlocal}.

Because the behavior of the solutions of \eqref{originalequation} consists of two equilibrium states joined together by a transition layer near the interface defined by \eqref{effectiveequation}, and the effective Hamiltonian $\overline{H}(p)$ can be interpreted as the propagation speed of this interface, our work is connected with the well-studied areas of traveling wave solutions of the KPP equation and the speed of their associated traveling fronts. There have been some recent developments in these areas in the nonlocal case. Coville, D\'avila, and Mart\'inez studied \eqref{originalequation} for periodic $f$ in \cite{covilledavilamartinez1} and \cite{covilledavilamartinez2} and showed that there exists a minimal speed, called the critical speed, for which there exists a pulsating front solution of \eqref{originalequation}. The existence of traveling wave solutions and a critical speed was considered for a non-local KPP equation similar to \eqref{originalequation} by Berestycki, Nadin, Perthame, and Ryzhik in \cite{berestycki}. Transition fronts for a one-dimensional space-inhomogeneous KPP equation were studied by Nolen, Roquejoffre, Ryzhik, and Zlatos in \cite{zlatos1}, where they gave conditions for the existence of transition fronts and found a range of possible speeds. This work was extended to \eqref{originalequation} by Lim and Zlatos in \cite{zlatos2}.

The local version of (\ref{originalequation}), i.e. the equation where the integral term is replaced by a uniformly elliptic second-order operator, has been studied extensively. Its rescaled form reads
\begin{equation} \label{localequation}
u^\epsilon_t - \epsilon a_{ij}\left(x, \frac{x}{\epsilon}\right)u^\epsilon_{ij} + \epsilon^{-1} f\left(x, \frac{x}{\epsilon}, u^\epsilon\right) = 0.
\end{equation}
It was originally studied in the 1930's by Fisher in \cite{fisher} and by Kolmogorov, Petrovskii, and Piskunov in \cite{kpp}. Freidlin in \cite{freidlin} studied the behavior of (\ref{localequation}) using probabilistic methods for the $\frac{x}{\epsilon}$-independent problem. Evans and Souganidis in \cite{souganidisevans} extended \cite{freidlin} and introduced a different approach based on PDE methods which has proven to be more flexible. The behavior of the $u^\epsilon$ in the presence of periodic space-time oscillation was analyzed by Majda and Souganidis \cite{majda}. Finally, Souganidis \cite{takis1999} and Lions and Souganidis \cite{lionssouganidis} considered stationary ergodic coefficients. Our work is an extension of this work, in particular \cite{souganidisevans} and \cite{majda}, to the case of a nonlocal diffusion. There is also a vast literature dealing with the long-time behavior of \eqref{localequation}, going back to the work of Aronson and Weinberger \cite{aronsonweinberger}.

Due to the presence of the oscillatory variable $\frac{x}{\epsilon}$ in \eqref{uequation}, the theory of homogenization plays a crucial part in the analysis of this equation as $\epsilon \rightarrow 0$. The study of homogenization of Hamilton-Jacobi equations in periodic settings began with the work of Lions, Papanicolaou, and Varadhan \cite{lpv}, and homogenization for ``viscous'' Hamilton-Jacobi equations was studied by Evans \cite{evansptf}. Arisawa in \cite{arisawa2} studied the periodic homogenization of integro-differential equations with L\'evy operators, equations that bear similarities to the ones we study.
Homogenization in the almost-periodic case was established by Ishii \cite{ishii}. In the stationary ergodic case, the homogenization of Hamilton-Jacobi equations was established by Souganidis \cite{takis1999} (see also Rezakhanlou and Tarver \cite{tarver}), while the viscous Hamilton-Jacobi equation was studied by Lions and Souganidis \cite{lionssouganidis} and by Kosygina, Rezakhanlou, and Varadhan \cite{kosygina}.
The main additional difficulty in the general stationary ergodic setting, as opposed to the almost periodic or periodic settings, is that a ``corrector'' function solving the macroscopic problem, or ``cell problem,'' cannot be found in general, as discussed by Lions and Souganidis in \cite{lionssouganidis1}. It was necessary to introduce a different approach based on the subadditive ergodic theorem and the control interpretation of the equation, and this restricts us to the study of convex Hamilton-Jacobi and viscous Hamilton-Jacobi equations.

Recently, Armstrong and Souganidis in \cite{takispaper} and \cite{takispaper2} put forward a systematic way to prove homogenization of convex Hamilton-Jacobi and viscous Hamilton-Jacobi equations which this paper employs, because homogenization implies the desired asymptotic behavior for solutions of \eqref{uequation}. This approach uses uniform \emph{a priori} bounds on the solutions of the approximated cell problem, in this case \eqref{approxcellproblemomega}, to find the effective Hamiltonian and a strictly sublinear at infinity ``subcorrector'' via weak limits, and then employs the ``metric problem'' in order to prove that homogenization occurs almost surely. However, such \emph{a priori} bounds are not readily available in the nonlocal case, unlike in the local case where Bernstein's method can be applied. A significant part of this paper is devoted to proving such estimates for \eqref{approxcellproblemomega}, which is done with an additional assumption on the inhomogeneity of $f$, namely that the oscillation of the quantity $f_u(x, 0)$ is less than the integral of $J$. This kind of condition is not unprecedented; it was shown in \cite{zlatos2} that an oscillation bound for $f_u$ is necessary for transition fronts to exist for \eqref{originalequation}. An additional difficulty in our setting lies in the adaptation of the metric problem, in particular with the issue of constructing a barrier to ensure that the metric problem is well-posed. This is circumvented by allowing solutions of the metric problem to have a jump discontinuity at the boundary; this discontinuity does not affect the homogenization of the metric problem.

The paper is organized as follows. In Section \ref{sec:preliminaries}, we make precise our assumptions and state the ergodic and subadditive ergodic theorems which will be used. In Section \ref{sec:maintheorem} we state our main result, Theorem \ref{maintheorem}, and give a heuristic justification for it. In Section \ref{sec:estimates}, the needed $L^\alpha$ and oscillation estimates are proven for solutions of the approximated cell problem \eqref{approxcellproblemomega}. The effective Hamiltonian is identified and its basic properties are studied in Section \ref{sec:effhamiltonian}. In Section \ref{sec:metricproblem} we consider the metric problem and obtain its almost sure homogenization using the subadditive ergodic theorem. In Section \ref{sec:homogproof} we conclude by proving the main homogenization result in the stationary ergodic setting.
\section{Preliminaries and Assumptions}
\label{sec:preliminaries}
The following assumptions will be in force throughout this paper. We assume that $f\in C^\infty(\mathbb{R}^{n+1})$, satisfies
\begin{equation} \label{fassumption}
\sup_{x \in \mathbb{R}^n, |u| \leq L} \{|D_xf(x, u)| + |D^2_xf(x, u)|\} < \infty \text{ for each } L > 0,
\end{equation}
and is of KPP type. That is, for every $x \in \mathbb{R}^n$, $f$ satisfies
\begin{align} \label{monostable}
&\left\{
\begin{array}{rl}
f(x, u) < 0 & \text{for } u \in (-\infty, 0) \cup (1, \infty),\\
f(x, u) > 0 & \text{for } u \in (0, 1), \end{array}\right. \\
\label{kppcondition}
&c(x) := \frac{\partial f}{\partial u}(x, 0)= \sup_{u > 0} u^{-1} f(x, u) \geq \kappa > 0.
\end{align}
Note that due to (\ref{fassumption}), $c(x)$ is smooth, bounded, and Lipschitz continuous with constant $K$.

Concerning the kernel $J$, we assume that
\begin{equation} \label{Jassumption}
\left\{
\begin{array}{l}
 J \text{ is compactly supported in a set } O \subset B(0, \overline{r}), \\
 J \in C(\mathbb{R}^n), J(x) = J(-x) \text{ for all } x \in \mathbb{R}^n, \int_{\mathbb{R}^n} J(y) dy = \bar{J} < \infty,\\
 \text{There exists } r_1 > 0 \text{ such that } J(y) \geq A > 0 \text{ on } B(0, r_1).\end{array} \right.
\end{equation}
Concerning the initial condition $u_0$, we assume that
\begin{equation} \label{u0assumption}
u_0 \in C(\mathbb{R}^n), 0 \leq u_0 \leq 1, \text{ and } G_0 = \mathrm{supp}(u_0) \text{ is compact}.
\end{equation}
In general, our problem is considered in a random environment described by a probability space $(\Omega, \mathcal{F}, \mathbb{P})$, which is endowed with a group $(\tau_y)_{y \in \mathbb{R}^n}$ of $\mathcal{F}$-measurable, measure-preserving transformations $\tau_y: \Omega \rightarrow \Omega$. We assume that $(\tau_y)$ is \emph{ergodic}; that is, if $D \subseteq \Omega$ is a set that satisfies $\tau_z(D) = D$ for all $z \in \mathbb{R}^n$, then either $\mathbb{P}[D] = 0$ or $\mathbb{P}[D] = 1$. An $\mathcal{F}$-measurable process $f$ is said to be \emph{stationary} if for all $y, z \in \mathbb{R}^n$ and $\omega \in \Omega$, $f(y, \tau_z \omega) = f(y+z, \omega).$ Our assumption on the coefficient $c(z, \omega)$ in the stationary ergodic setting is that
\begin{equation} \label{stationarityassumption}
c(z, \omega) \text{ is stationary}.
\end{equation}
It can be seen that 1-periodicity and almost-periodicity are specific cases of the stationarity assumption.

It is necessary to prove some \emph{a priori} bounds, and for this equation, we have not been able to do so without controlling the oscillation of the function $c$:
\begin{equation} \label{c0assumption}
\osc_{\mathbb{R}^n} c(z, \omega) := \sup_{\mathbb{R}^n} c(z, \omega) - \inf_{\mathbb{R}^n} c(z, \omega) = \rho < \bar{J}.
\end{equation}
As mentioned in the introduction, a similar assumption is necessary for transition fronts to exist for \eqref{originalequation} and its local analogue (see \cite{zlatos1}, \cite{zlatos2}).

The following ergodic theorem can be found in Becker \cite{becker} and will be used several times in this paper.
\begin{proposition}
Suppose that $f: \mathbb{R}^n \times \Omega \rightarrow \mathbb{R}$ is stationary and $\mathbb{E}[|f(0, \cdot)|] < \infty$. Then there exists a subset $\hat{\Omega} \subseteq \Omega$ of full probability such that for each bounded domain $V \subset \mathbb{R}^n$ and $\omega \in \hat{\Omega}$,
$$
\lim_{t \rightarrow \infty} \dashint_{tV} f(y, \omega) dy = \mathbb{E}[f].
$$
\end{proposition}
We will also employ a subadditive ergodic theorem in this paper, and its statement will require some additional notation. Let $\mathcal{I}$ denote the class of subsets of $[0, \infty)$ which consists of finite unions of intervals of the form $[a, b)$. Let $\{\sigma_t\}_{t \geq 0}$ be a semigroup of measure-preserving transformations on $\Omega$. A \emph{continuous subadditive process} on $(\Omega, \mathcal{F}, \mathbb{P})$ with respect to $\sigma_t$ is a map $Q: \mathcal{I} \rightarrow L^1(\omega, \mathbb{P})$ that satisfies the following conditions:
\begin{compactenum}
    \item $Q(I)(\sigma_t \omega) = Q(t+I)(\omega)$ for each $t > 0$, $I \in \mathcal{I}$ and a.s. in $\omega$,
    \item $\mathbb{E}[|Q(I)|] \leq C |I|$ for some $C > 0$ and every $I \in \mathcal{I}$,
    \item If $I_1, I_2, \ldots, I_k \in \mathcal{I}$ are disjoint, then $\displaystyle Q\left(\bigcup_{j = 1}^k I_j\right) \leq \sum_{j = 1}^k Q(I_j).$
\end{compactenum}
For the proof of the following subadditive ergodic theorem, see Akcoglu and Krengel \cite{akcoglukrengel}.
\begin{proposition}
Suppose that $Q$ is a continuous subadditive process. Then there is a random variable $a(\omega)$ such that
$$
\lim_{t \rightarrow \infty} \frac{1}{t} Q([0, t))(\omega) = a(\omega) \text{ a.s. in } \omega.
$$
If $\{\sigma_t\}_{t > 0}$ is ergodic, then $a(\omega)$ is deterministic, that is, constant with respect to $\omega$.
\end{proposition}
\section{Main Theorem, Heuristic Derivation}
\label{sec:maintheorem}
We now state our main result.
\begin{theorem} \label{maintheorem}
Assume \eqref{fassumption}-\eqref{c0assumption}. Then almost surely in $\omega$ there exists a continuous function $\overline{H}: \mathbb{R}^n \rightarrow \mathbb{R}$ such that as $\epsilon \rightarrow 0$, $u^\epsilon \rightarrow 0$ in $\{\phi < 0\}$ and $u^\epsilon \rightarrow 1$ in $\mathrm{int}\{\phi = 0\}$ locally uniformly, where $\phi$ is the unique solution of \eqref{effectiveequation}.
\end{theorem}
Next we explain in a heuristic way the origin of the variational inequality and why it controls the asymptotic behavior of the $u^\epsilon$. Following the work for local KPP equations mentioned in the introduction, we now use the classical Hopf-Cole transformation
\begin{equation} \label{hopfcole}
u^\epsilon = \exp(\epsilon^{-1} \phi^\epsilon).
\end{equation}
It is immediate that for $t = 0$, $\phi^\epsilon = -\infty$ on $\mathbb{R}^n \backslash \overline{G}_0$ and $\phi^\epsilon \rightarrow 0$ on $G_0$ as $\epsilon \rightarrow 0.$ The interesting part of the transformation comes into play for $t > 0$. We can see via straightforward calculations that $\phi^\epsilon$ solves
$$
\phi^\epsilon_t(x, t) + \bar{J} - \int J(y) \exp\left(\frac{\phi^\epsilon(x - \epsilon y, t) - \phi^\epsilon(x, t)}{\epsilon}\right) dy - \frac{1}{u^\epsilon} f\left(\frac{x}{\epsilon}, u^\epsilon\right) = 0,
$$
an equation which can be analyzed using homogenization techniques. We assume that $\phi^\epsilon$ admits the asymptotic expansion $\phi^\epsilon(x, t) = \phi(x, t) + \epsilon v(\frac{x}{\epsilon}) + O(\epsilon^2).$ Writing $z = \frac{x}{\epsilon}$ and performing a formal computation, we obtain
$$
\phi_t + \bar{J} - \int J(y) \exp\left(\frac{\phi(x - \epsilon y, t) -\phi(x, t)}{\epsilon} + v(z-y) - v(z))\right) dy - \frac{1}{u^\epsilon} f\left(z, u^\epsilon\right)  = 0.
$$
Formally, we can say that as $\epsilon \rightarrow 0$, $\epsilon^{-1}(\phi(x - \epsilon y, t) -\phi(x, t)) \rightarrow -y\cdot D\phi(x, t).$ In addition, if $u^\epsilon \rightarrow 0$ as $\epsilon \rightarrow 0$, then
\begin{equation}
(u^\epsilon)^{-1} f(z, u^\epsilon) \rightarrow \displaystyle \frac{\partial f}{\partial u}(z, 0) = c(z). \label{blah2}
\end{equation}
Writing $p = D\phi(x, t)$, we see that oscillatory behavior disappears in the limit as $\epsilon \rightarrow 0$ if it is possible to find a constant $\overline{H}(p)$ and a function $v$ that solves
\begin{equation} \label{cellproblem}
\bar{J}-\int J(y) \exp(-y\cdot p) \exp(v(z-y) - v(z)) dy - c(z) = \overline{H}(p),
\end{equation}
which is a typical macroscopic problem or ``cell problem'' from homogenization theory. The issue is to find $\overline{H}(p)$, referred to as the effective Hamiltonian, so that (\ref{cellproblem}) admits a solution $v$, typically referred to as a ``corrector,'' with appropriate behavior at infinity i.e. strict sublinearity, so that $\overline{H}(p)$ is unique. If an effective Hamiltonian and a corresponding corrector can be found, then we see that $\phi^\epsilon$ converges to a function $\phi$ that satisfies $\phi_t + \overline{H}(D\phi) = 0,$ provided that we also ensure that $\phi < 0$ so $u^\epsilon \rightarrow 0$ due to \eqref{hopfcole}, which would then allow us to apply \eqref{blah2}.
Therefore, $\phi$ should satisfy the Hamilton-Jacobi variational inequality
$$
\max(\phi_t + \overline{H}(D\phi), \phi) = 0,
$$
which combined with the initial condition at $t = 0$ is precisely \eqref{effectiveequation}. Then \eqref{hopfcole}, the fact that $\phi^\epsilon \rightarrow \phi$, and an additional argument to show that $u^\epsilon \rightarrow 1$ on the set $\{\phi = 0\}$ imply that $u^\epsilon$ satisfies the behavior described by Theorem \ref{maintheorem}.
\section{Estimates for the Approximate Cell Problem}
\label{sec:estimates}
As indicated in Section \ref{sec:maintheorem}, the key to proving Theorem \ref{maintheorem} is to find an effective Hamiltonian $\overline{H}(p)$ so that appropriate correctors exist. This problem in the case when $f$ is an almost periodic case was solved in \cite{almostperiodicpaper}, but the techniques of that work cannot be applied directly due to the lack of compactness in the stationary ergodic setting. We will use the framework of \cite{takispaper}, which defines $\overline{H}(p)$ as a weak limit of $\lambda v^\lambda(0, \omega)$, where $v^\lambda$ is the solution of the approximate cell problem
\begin{equation}
\label{approxcellproblemomega}
\lambda v^\lambda(z, \omega) + \bar{J} -\int J(y) \exp(-y\cdot p) \exp(v^\lambda(z-y, \omega) - v^\lambda(z, \omega)) dy - c(z, \omega) = 0,
\end{equation}
and then finds a strictly sublinear at infinity (approximate) subcorrector also by a weak limit. This approach requires \emph{a priori} estimates on $v^\lambda$. In \cite{takispaper} the estimates are obtained by Bernstein's method, but new methods are required for this nonlocal setting. Therefore our objective in this section is to prove local oscillation bounds for solutions of this equation that are uniform in $\lambda$ and $\omega$.
\begin{proposition} \label{osctheorem}
Assume \eqref{fassumption}-\eqref{Jassumption}, \eqref{stationarityassumption}, \eqref{c0assumption}. If $\omega \in \Omega$ and $R \geq \frac{r_1}{2}$, then
\begin{equation} \label{boundedoscbound}
\osc_{B(0, R)} v^\lambda(\cdot, \omega) \leq CR,
\end{equation}
where $C= C(K, p, \rho)$ is uniform in $\lambda, \omega$, and $R$. In particular, for any $R > 0$,
\begin{equation} \label{firstoscbound}
\limsup_{\lambda \rightarrow 0} \osc_{B(0, \frac{R}{\lambda})} \lambda v^\lambda(\cdot, \omega) \leq  CR.
\end{equation}
\end{proposition}
It can be shown that \eqref{approxcellproblemomega} satisfies a comparison principle and is well-posed; see \cite{almostperiodicpaper} for a proof. Note that $v^\lambda$ is stationary; this is true by \eqref{stationarityassumption} and the fact that comparison holds for (\ref{approxcellproblemomega}). The bounds will be proven $\omega$-by-$\omega$; that is, we show that the bounds hold for each fixed $\omega \in \Omega$, so in this section we will suppress $\omega$ for notational convenience. \cite{takispaper} uses Bernstein's method and Sobolev inequalities to prove the needed bounds. Here we use an entirely different approach, taking into consideration the nonlocal character of the equation. Define $w^\lambda(z) := v^\lambda(z) - v^\lambda(0).$ We know that $w^\lambda$ satisfies
\begin{equation}
\lambda (w^\lambda(z) + v^\lambda(0)) + \bar{J} - \int_{B_{\overline{r}}} J(y)\exp(-y\cdot p) \exp(w^\lambda(z-y) - w^\lambda(z)) dy - c(z) = 0 \label{wlequation}
\end{equation}
Recall from \eqref{Jassumption} that the support of $J$ is contained in $B(0, \overline{r})$, and that $r_1$ is a radius such that $J(y) \geq A > 0$ on $B(0, r_1)$. We define some additional notation. Let $\Phi^\lambda(z)$ be defined to be the nonlocal term in (\ref{wlequation}):
\begin{equation} \label{philambda}
\Phi^\lambda(z) := \int_{B_{\overline{r}}} J(y) \exp(-y \cdot p) \exp(w^\lambda(z-y) - w^\lambda(z)) dy.
\end{equation}
We will also consider the quantity
\begin{equation} \label{psidefinition}
\Psi^\lambda(z):= \int_{B_{\overline{r}}} J(y)\exp(-y\cdot p)\exp(w^\lambda(z-y)) dy = \Phi^\lambda(z) \exp(w^\lambda(z)).
\end{equation}
The first step in proving oscillation bounds is to show that $w^\lambda$ is $L^\alpha_{loc}$ uniformly bounded, and the key property we will use is to show that $\Phi^\lambda$ must be uniformly bounded in $\lambda$ (depending on $p$) because by comparison the other terms are uniformly bounded.
\begin{proposition} \label{lptheorem}
Under the assumptions of Proposition \ref{osctheorem}, $\|w^\lambda\|_{L^\alpha(B_R)} \leq C$ for any $R > 0$ and for any $1 \leq \alpha < \infty$, where $C$ depends on $R, p$ but is uniform in $\lambda$.
\end{proposition}
\begin{proof}
The proposition follows if we can show that $\|w_\lambda\|_{L^\alpha(B(z, \frac{r_1}{2}))}$ is uniformly bounded in $\lambda$ for any $z$ such that $O_1 := B(z, \frac{r_1}{2}) \subset B_R$, because we can cover $B_R$ with finitely many such balls. We will now show that the positive and negative parts of $w^\lambda$ are uniformly $L^\alpha(O_1)$ bounded in Lemmas \ref{positivecase} and \ref{negativecase}.
\begin{lemma} \label{positivecase}
$\|w^{\lambda, +}\|_{L^\alpha(O_1)}$ is bounded uniformly in $\lambda$.
\end{lemma}
\begin{proof}
Suppose for a contradiction that $\|w^{\lambda, +}\|_{L^\alpha(O_1)} \rightarrow \infty$. For each $\lambda$, define $z^\lambda_{min}$ to be the point in $\overline{O_1}$ where $w^\lambda$ is minimized. We can now show that $w^\lambda(z^\lambda_{min}) \rightarrow \infty$, because if there were a subsequence where $w^\lambda(z^\lambda_{min}) \leq C$, then we would have that $\Phi^\lambda(z^\lambda_{min}) \rightarrow \infty$ as $\lambda \rightarrow 0$, because
\begin{align*}
\Phi^\lambda(z^\lambda_{min}) &= \frac{\Psi^\lambda(z^\lambda_{min})}{w^\lambda(z^\lambda_{min})} \geq \frac{1}{C} \int_{B_{r_1}} J(y)\exp(-y\cdot p)\exp(w^\lambda(z^\lambda_{min}-y)) dy \\
&\geq C_2 \int_{O_1} (w^{\lambda, +}(y))^\alpha dy \rightarrow \infty.
\end{align*}
The second inequality holds because $\exp(-y \cdot p)$ and $J(y)$ are bounded below away from zero on $B(0, r_1)$, and because for any $\alpha$, there exists a constant $K$ such that $\exp(x) \geq Kx^\alpha$ holds for all $x$.
This means that $\Phi^\lambda(z^\lambda_{min}) \rightarrow \infty$, which is a contradiction because $\Phi^\lambda$ must be uniformly bounded, so $w^\lambda(z^\lambda_{min}) \rightarrow \infty$, which implies that $\displaystyle \min_{\overline{O}_1} w^\lambda \rightarrow \infty$.

Now consider $O_2 := B(x_1, \frac{r_1}{2}) \subset B_R$, such that $O_1$ and $O_2$ overlap nontrivially, that is, $|O_1 \cap O_2| > 0$. Define $z^\lambda_{min, 2}$ to be the point in $\overline{O}_2$ where $w^\lambda$ is minimized. We claim that $w^\lambda(z^\lambda_{min, 2}) \rightarrow \infty$, and this follows in a very similar fashion to the preceding argument. Suppose for a contradiction that $w^\lambda(z^\lambda_{min, 2}) \leq C_2$ along a subsequence also called $\lambda$. Then because $O_1 \cap O_2 \subset B(z^\lambda_{min}, r_1)$, we have that
\begin{align*}
\Phi^\lambda(z^\lambda_{min, 2}) &= \frac{\Psi^\lambda(z^\lambda_{min, 2})}{\exp(w^\lambda(z^\lambda_{min, 2}))} \geq \frac{1}{C_2} \int_{B_{\overline{r}}} J(y)\exp(-y\cdot p)\exp(w^\lambda(z^\lambda_{min, 2}-y))) dy \\
&\geq C \int_{O_1 \cap O_2} \exp(w^\lambda(y)) dy \geq C \exp(\min_{\overline{O}_1}(w^\lambda))|O_1 \cap O_2| \rightarrow \infty,
\end{align*}
because $\displaystyle \min_{\overline{O}_1}(w^\lambda) \rightarrow \infty$. This is again a contradiction. Therefore $\displaystyle \min_{\overline{O}_2} w^\lambda \rightarrow \infty$ as $\lambda \rightarrow 0$. However, we can repeat this process finitely many times by taking a ball $O_3$ overlapping $O_2$ and so on, until the ball in question contains the origin. This yields $w^\lambda(0) \rightarrow \infty$, but $w^\lambda(0) = 0$ for all $\lambda$ by construction. Therefore we have reached a contradiction, and so $\|w^{\lambda, +}\|_{L^\alpha(O_1)}$ is uniformly bounded.
\end{proof}
\begin{lemma} \label{negativecase}
$\|w^{\lambda, -}\|_{L^\alpha(O_1)}$ is uniformly bounded in $\lambda$.
\end{lemma}
\begin{proof}\let\qed\relax
Assume for a contradiction that $\|w^{\lambda, -}\|_{L^\alpha(O_1)} \rightarrow \infty$.  First we prove a technical lemma, which essentially states that on any positive measure subset of a ball of radius $\frac{r_1}{2}$, the minimum of $w^\lambda$ over the subset cannot deviate too far from the minimum of $w^\lambda$ over the ball of radius $\frac{r_1}{2}$.
\begin{lemma} \label{minimumlemma}
Consider a ball $U = B(z, \frac{r_1}{2})$ for $z \in B(0, R)$ such that $U \subset B(0, R)$, and define $\displaystyle C_\lambda = \min_{\overline{U}} w^\lambda(\cdot)$. Then for any set $U_1 \subset U$ of positive measure, $\displaystyle [\min_{\overline{U}_1} w^\lambda(\cdot)] - C_\lambda \leq C(U_1)$, where $C(U_1)$ is uniform in $\lambda$.
\end{lemma}
\begin{proof}
Define $z^\lambda_{min}$ to be point where the minimum of $w^\lambda$ over $\overline{U}$ is achieved. If there existed $U_1$ such that $\displaystyle [\min_{\overline{U_1}} w^\lambda(\cdot)] - C_\lambda \rightarrow \infty$ (along a subsequence also called $\lambda$), then we can evaluate $\Phi^\lambda$ at $z^\lambda_{min}$ to get
\begin{align*}
\Phi^\lambda(z^\lambda_{min}) &= \int_{B_{\overline{r}}} J(y)\exp(-y\cdot p) \exp(w^\lambda(z^\lambda_{min}-y)-w^\lambda(z^\lambda_{min})) \\
&\geq C\int_{U_1} \exp(w^\lambda(y)-w^\lambda(z^\lambda_{min})) dy \geq \exp((\min_{\overline{U}_1} w^\lambda(\cdot))- w^\lambda(z^\lambda_{min}))|U_1|  \rightarrow \infty,
\end{align*}
as $\lambda \rightarrow 0$, where once again $C$ is a positive constant, and this is again a contradiction to the boundedness of $\Phi^\lambda$.
\end{proof}
Next, we will use Lemma \ref{minimumlemma} to show that $\Psi^\lambda(z) \rightarrow 0$ as $\lambda \rightarrow 0$.
\begin{lemma} \label{psilemma}
Under the assumption $\|w^{\lambda, -}\|_{L^\alpha(O_1)} \rightarrow \infty$, $\Psi^\lambda(z) \rightarrow 0$ for all $z \in B(0, R)$ as $\lambda \rightarrow 0$.
\end{lemma}
\begin{proof}
We first use Lemma \ref{minimumlemma} to show that for any ball $U = B(z, \frac{r_1}{2}) \subset B(0, R)$,
\begin{equation} \label{lemma44claim}
\min_{\overline{U}} w^\lambda(\cdot) \rightarrow -\infty \text{ as } \lambda \rightarrow 0.
\end{equation}
This follows by an overlapping balls argument. Because $\|w^{\lambda, -}\|_{L^\alpha(O_1)} \rightarrow \infty$, we know that $C_\lambda := \min_{\overline{O}_1} w^\lambda(\cdot) \rightarrow -\infty \text{ as } \lambda \rightarrow 0.$ Consider a ball $O_2$ overlapping $O_1$ so that $|O_1 \cap O_2| > 0$. By Lemma \ref{minimumlemma},
$$
\min_{\overline{O_1 \cap O_2}} w^\lambda(\cdot) - C_\lambda \leq C < \infty \text{ uniformly in } \lambda.
$$
Since $C_\lambda \rightarrow -\infty$ as $\lambda \rightarrow 0$, then $\displaystyle \min_{\overline{O_1 \cap O_2}} w^\lambda(\cdot) \rightarrow -\infty,$ which implies that $\min_{\overline{O}_2} w^\lambda(\cdot) \rightarrow -\infty.$ However, this argument can be applied finitely many times using a ball $O_3$ overlapping $O_2$, etc, until our final ball is $U$, so we have shown \eqref{lemma44claim}, which we will now use to prove Lemma \ref{psilemma}. Suppose for a contradiction that for some $z$ and some subsequence of $\lambda$, $\Psi^\lambda(z) \rightarrow C_3 > 0$. 
Then by the symmetry of $J$, there exists another subsequence (denoted $\lambda$) and a ball $U_1 \subset B(z, \overline{r}) \cap B(0, R)$ of radius $\frac{r_1}{2}$ such that
$$
\liminf_{\lambda \rightarrow 0} \int_{U_1} \exp(w^\lambda(y)) dy \geq C_4 > 0.
$$
Define $z^\lambda_{min}$ to be the point where $w^\lambda$ is minimized over $\overline{U}_1$, and evaluate $\Phi^\lambda$ at $z^\lambda_{min}$ to get
\begin{align*}
\Phi^\lambda(z^\lambda_{min}) &= \int_{B_{\overline{r}}} J(y)\exp(-y\cdot p)\exp(w^\lambda(z^\lambda_{min}-y) - w^\lambda(z^\lambda_{min})) dy \\
&\geq C \int_{U_1} \exp(w^\lambda(y) - w^\lambda(z^\lambda_{min})) dy \\
&\geq \frac{\int_{U_1}J(y-z^\lambda_{min})\exp(-(y-z^\lambda_{min}) \cdot p)\exp(w^\lambda(y)) dy}{\exp(w^\lambda(z^\lambda_{min}))} \\
&\geq \frac{C\int_{U_1} \exp(w^\lambda(y))dy}{\exp(w^\lambda(z^\lambda_{min}))} \geq \frac{CC_4}{\exp(w^\lambda(z^\lambda_{min}))} \rightarrow \infty
\end{align*}
by \eqref{lemma44claim}, and this is again a contradiction.
\end{proof}
Note that $\Psi^\lambda(0) = \Phi^\lambda(0)$ by construction, and we will now use this fact and Lemma \ref{psilemma} to finish the proof. If we evaluate (\ref{wlequation}) at $z = 0$, because $w^\lambda(0) = 0$ by construction, we have $\lambda v^\lambda(0) + \bar{J} - \Psi^\lambda(0) - c(0) = 0.$ We can then write $v^\lambda(0) = \frac{1}{\lambda}(c(0) - \bar{J} + \Psi^\lambda(0)).$ Now we consider (\ref{wlequation}) with this value of $v^\lambda(0)$, and show that the function $\varphi(z) := 0$ is a subsolution of (\ref{wlequation}). Inserting $\varphi$ into (\ref{wlequation}) in place of $w^\lambda$, we have that the left hand side is
\begin{equation} \label{lhs}
\lambda (v^\lambda(0) + 0) + \bar{J} - \int_{B(0, \overline{r})} J(y) \exp(-y \cdot p)dy - c(z) \leq c(0) - \bar{J} + \Psi^\lambda(0) - c(z),
\end{equation}
where the inequality follows due to our expression for $v^\lambda(0)$ and the symmetry of $J$.
Then because $\Psi^\lambda(0) \rightarrow 0$ by Lemma \ref{psilemma}, we have that for $\lambda$ sufficiently small (\ref{lhs}) is nonpositive by (\ref{c0assumption}). Therefore, by the comparison principle for (\ref{wlequation}), $w^\lambda \geq \varphi \equiv 0$ for sufficiently small $\lambda$, which means that $(w^\lambda)^-$ cannot be unbounded in $L^\alpha(O_1)$. This finishes the proof of Lemma \ref{negativecase}.
\end{proof}
The combination of Lemmas \ref{positivecase} and \ref{negativecase} imply Proposition \ref{lptheorem}.
\end{proof}
In this proof, we have actually proven a stronger statement than Proposition \ref{lptheorem}; we have actually proven that $\Psi^\lambda$ is locally uniformly bounded above and bounded below away from zero. The proof is very similar to the proof of Proposition \ref{lptheorem}, so we will state the proposition and sketch the proof.
\begin{proposition} \label{psiboundprop}
Let $R \geq \frac{r_1}{2}$. Then under the assumptions of Proposition \ref{osctheorem}, for $z \in B(0, R)$,
\begin{equation} \label{uniformpsibound}
C^{-1} \leq \Psi^\lambda(z) \leq C
\end{equation}
uniformly in $\lambda$ and $z$, for some constant $C = C(|p|) > 1$ that satisfies $\log(C) \leq C'R$, where $C'$ is a constant independent of $R$.
\end{proposition}
\begin{proof}
The proof of Proposition \ref{lptheorem} implies that there exists a uniform-in-$\lambda$ constant $C$ such that $C^{-1} \leq \Psi^\lambda(0) \leq C.$ We now consider $\Psi^\lambda(z)$ for $z \in B(0, \frac{r_1}{2})$. We first show that
\begin{equation}
\Psi^\lambda(z) \leq CC_2, \label{blah11}
\end{equation}
for a constant $C_2$ uniform in $\lambda$. Suppose there were no such constant $C_2$. Then there exists a subsequence of $\lambda$ such that $\Psi^\lambda(z) \rightarrow \infty$ as $\lambda \rightarrow \infty$. We can then repeat the argument of Lemma \ref{positivecase} to reach a contradiction, noting that $\Phi^\lambda$ is uniformly bounded above by $K$, $J$ is uniformly bounded below by $A$ on $B(0, \frac{r_1}{2})$, and $|B(z, \frac{r_1}{2}) \cap B(0, \frac{r_1}{2})| \geq \frac{\pi r_1^2}{16},$ all constants which are independent of $\lambda$, which means that $C_2$ is also independent of $\lambda$. We can then show that for $z \in B(0, r_1) \backslash B(0, \frac{r_1}{2})$, $\Psi^\lambda(z) \leq CC_2^2,$ which is proven by considering the overlap of $B(z, \frac{r_1}{2})$ with a ball of the form $B(z_1, \frac{r_1}{2})$ with $z_1 \in B(0, \frac{r_1}{2})$, because we know that $\Psi^\lambda(z_1) \leq CC_2$ by \eqref{blah11}. This procedure can be repeated for $z \in B(0, \frac{3r_1}{2}) \backslash B(0, r_1)$, and so forth. We can write for an arbitrary $R \geq \frac{r_1}{2}$ that $\Psi^\lambda_z \leq CC_2^L, L = \lceil\frac{2R}{r_1}\rceil,$ which gives us half of Proposition \ref{psiboundprop}.

We next show that $\Psi^\lambda(z) \geq C^{-1}C_3,$ for a constant $C_3$ uniform in $\lambda$. Suppose that there does not exist such a constant. Then there exists a subsequence of $\lambda$ such that $\Psi^\lambda(z) \rightarrow 0$ as $\lambda \rightarrow \infty$. We can now repeat the argument of Lemma \ref{negativecase} to reach a contradiction, again noting that the relevant constants, this time with the addition of $\rho$ from \eqref{c0assumption}, do not depend on $\lambda$, which implies that $C_3$ is independent of $\lambda$. Then, we can show that for $z \in B(z, r_1) \backslash B(0, \frac{r_1}{2})$, $\Psi^\lambda(z) \geq C^{-1}C_3^2.$ This can be shown using the overlap of $B(z, \frac{r_1}{2})$ with a ball $B(z_1, \frac{r_1}{2})$ with $z_1 \in B(0, \frac{r_1}{2})$. We can now continue this process to prove an estimate for $\Psi(z)$ for $z \in B(0, \frac{3r_1}{2}) \backslash B(0, r_1)$, and so forth, until we have covered $B(0, R)$. This finishes the proof of \eqref{uniformpsibound}.
\end{proof}
We will now prove Proposition \ref{osctheorem}. We show that \eqref{boundedoscbound} holds for $w^\lambda$ because this is equivalent to showing that the same inequality holds for $v^\lambda$.
\begin{proof}[Proof of Proposition \ref{osctheorem}]
We know that by Proposition \ref{psiboundprop}, \eqref{uniformpsibound} holds uniformly in $\lambda$. Therefore, since $\Phi^\lambda(z)\leq K$, \eqref{psidefinition} implies that $w^\lambda(z) \geq -CR$ holds for all $z$ and all $\lambda$, and so $\inf_{B(0, R)} w^\lambda(\cdot) \geq -CR.$

To prove a corresponding upper bound, suppose for a contradiction that there exists sequences $z_n \in B(0, R)$ and $\lambda_n$ such that $w^{\lambda_n}(z_n) \rightarrow \infty$. We can use the same argument as in the proof of Proposition 4.3 of \cite{almostperiodicpaper} to show that $\lambda w^\lambda(z_n) \rightarrow 0$, because here we have that $z_n$ is contained in a bounded set, which is what was needed for that argument to hold. Combining this with \eqref{uniformpsibound} and evaluating (\ref{wlequation}) at $z_n$ for the subsequence $\lambda_n$, we get $\lambda v^{\lambda_n}(0) = c(z_n) - 1 + o_\lambda(1),$ which means that due to (\ref{c0assumption}), we can take $\lambda_n \rightarrow 0$ to conclude that $\lambda v^\lambda(0) < \inf_{\mathbb{R}^n} c(\cdot),$
but this is a contradiction because the constant function
$$
\varphi(z) = \lambda^{-1} \inf_{\mathbb{R}^n} c(\cdot)
$$
is a subsolution of \eqref{approxcellproblemomega}. Because the constants involved are independent of $\lambda$, in fact we have $\sup_{B(0, R)} w^\lambda(\cdot) \leq CR.$ This gives us \eqref{boundedoscbound}.
\end{proof}
In addition, by defining $w^\lambda(z) := v^\lambda(z) - v^\lambda(\hat{z})$ for any $\hat{z} \in \mathbb{R}^n$ and using the same arguments used to prove Proposition \ref{osctheorem}, we can show that (\ref{boundedoscbound}) holds for balls centered at any $\hat{z} \in \mathbb{R}^n$. The constant is also uniform in $\hat{z}$ because it depends only on $\rho$ from \eqref{c0assumption}, $p$, and $K$.
\begin{theorem} \label{finalosctheorem}
Let $\hat{z} \in \mathbb{R}^n$. Under the assumptions of Proposition \ref{osctheorem}, if $R \geq \frac{r_1}{2}$ there exists a constant $C$ which is uniform in $\lambda, R, \hat{z}$ such that
\begin{equation} \label{finaloscbound1}
\osc_{B(\hat{z}, R)} v^\lambda(\cdot, \omega) \leq CR.
\end{equation}
In particular, for any $y \in \mathbb{R}^n$ and $R > 0$,
\begin{equation}\label{finaloscbound}
\limsup_{\lambda \rightarrow 0} \osc_{B(\frac{y}{\lambda}, \frac{R}{\lambda})} \lambda v^\lambda(\cdot, \omega) \leq CR.
\end{equation}
\end{theorem}
To conclude this section, we give a modulus of continuity estimate for $v^\lambda$. Because the proof is the same as Lemma 4.7 of \cite{almostperiodicpaper}, we omit the proof.
\begin{lemma} \label{plipschitzlemma}
There exist $C_3, C_4 > 0$ such that for each $\lambda > 0$, $p_1, p_2 \in \mathbb{R}^n$
\begin{equation} \label{plipschitz}
\sup_{y \in \mathbb{R}^n} |\lambda v^\lambda(y; p_1) - \lambda v^\lambda(y; p_2)| \leq C_3\exp(C_4(1+|p_1|+|p_2|))|p_1 - p_2|.
\end{equation}
\end{lemma}
\section{The Effective Hamiltonian $\overline{H}(p)$}
\label{sec:effhamiltonian}
Using the approach of \cite{lionssouganidis} and \cite{takispaper}, we will now identify the effective Hamiltonian $\overline{H}(p)$ in the stationary ergodic case as a weak limit by utilizing the previously derived oscillation bounds, and we will also show the existence of an ``approximate supercorrector'' that is strictly sublinear at infinity, i.e. a function that satisfies \eqref{supercorrector} for any $\nu > 0$. To do this we rely on the convexity of the exponential Hamiltonian and the lower semicontinuity of convex functionals with respect to weak convergence. In order to show that the approximate supercorrector is strictly sublinear at infinity, we use gradient bounds obtained by taking the mollification of the functions $w^\lambda$ from Section \ref{sec:estimates}. Note that in the almost periodic case, it is possible to find an ``approximate corrector'' for any $\nu > 0$ i.e. a function that satisfies the inequality
\begin{equation} \label{subcorrector}
1 - \int J(y) \exp(-y\cdot p) \exp(w_\nu(z-y, \omega) - w_\nu(z, \omega)) dy - c(z, \omega) \leq \overline{H}(p) + \nu
\end{equation}
in addition to \eqref{supercorrector}. However, the lack of compactness in the stationary ergodic setting prevents us from finding a function that satisfies both \eqref{subcorrector} and \eqref{supercorrector}.

\begin{proposition} \label{firstconvergenceprop}
There exists a set of full probability $\Omega_1 \subset \Omega$ and a continuous function $\overline{H}: \mathbb{R}^n \rightarrow \mathbb{R}$ such that for every $R > 0$, $p \in \mathbb{R}^n$, and $\omega \in \Omega_1$,
\begin{equation} \label{convergenceinEV}
\lim_{\lambda \rightarrow 0} \mathbb{E}\left[ \sup_{z \in B_{\frac{R}{\lambda}}} |\lambda v^\lambda(z, \cdot) + \overline{H}(p)|\right] = 0
\end{equation}
and
\begin{equation} \label{fulllimsup}
-\overline{H}(p) = \limsup_{\lambda \rightarrow 0} \lambda v^\lambda(0, \omega).
\end{equation}
Moreover, for each $p \in \mathbb{R}^d$ and $\nu > 0$, there exists a function $w_\nu: \mathbb{R}^n \times \Omega \rightarrow \mathbb{R}$ such that for every $\omega \in \Omega_1$ and $1 \leq \alpha < \infty$, $w_\nu(\cdot, \omega) \in W^{1, \alpha}_{\mathrm{loc}}(\mathbb{R}^n) \cap C^{0, 1}(\mathbb{R}^n)$, $Dw_\nu$ is stationary, and
\begin{align}
& 1 - \int J(y) \exp(-y\cdot p) \exp(w_\nu(z-y, \omega) - w_\nu(z, \omega)) dy - c(z, \omega) \geq \overline{H}(p) - \nu \text{ in } \mathbb{R}^n, \label{supercorrector} \\
&|z|^{-1} w_\nu(z, \omega) \rightarrow 0 \text{ as } |z| \rightarrow \infty. \label{strictlysublinear}
\end{align}
\end{proposition}
\begin{proof}
First we note that the continuity of $\overline{H}$ follows from (\ref{convergenceinEV}) and Lemma \ref{plipschitzlemma}, and that it suffices to prove the theorem for a fixed $p \in \mathbb{Q}^n$. Therefore in this proof we fix $p$ and omit dependence on it, and we will take $\Omega_1$ to be the intersection of the full probability sets that we construct for each rational $p$.

Because $c(z, \omega)$ is a bounded function, by comparison we have that $\lambda v^\lambda(0, \omega)$ is bounded in $L^\infty(\Omega)$. Proposition \ref{lptheorem} and Theorem \ref{finalosctheorem} state that $w^\lambda(z, \omega) := v^\lambda(z, \omega) - v^\lambda(0, \omega)$ is $L^\alpha_{loc}$ uniformly bounded in $\omega$ and $\lambda$, and that $w^\lambda$ and $v^\lambda$ satisfy (\ref{finaloscbound}). Now define
$$
w^\lambda_{\theta}(z, \omega) = (w^\lambda(\cdot, \omega) \ast \rho_\theta)(z),
$$
where $\ast$ denotes convolution and $\rho_\theta$ is a standard mollifier.  The local $L^\alpha$ bounds proven for $w^\lambda$ hold also for $w^\lambda_{\theta}$. Proposition \ref{finalosctheorem} in conjunction with the properties of mollification tell us that $\|Dw^\lambda_\theta\|_{\infty} \leq C(\theta)$ and $\|Dw^\lambda_\theta(\cdot, \omega)\|_{\alpha, loc} \leq C(\theta)$.

Therefore, we can find a subsequence $\lambda$, and for each $\theta$, we can find a subsequence $\lambda_j$ (taken to be a subsequence of $\lambda$), a random variable $\overline{H} = \overline{H}(p, \omega)$, a function $w_\theta \in L^\alpha_{loc}(\mathbb{R}^n \times \Omega)$ and a field $\Phi_\theta \in L^\alpha_{loc}(\mathbb{R}^n \times \Omega; \mathbb{R}^n)$ such that for every $R > 0$, as $\lambda \rightarrow 0$ and $j \rightarrow \infty$,
\begin{equation} \label{weakconvergence}
\left\{\begin{array}{ll}
-\lambda v^{\lambda}(0, \cdot) \rightharpoonup \overline{H}(p, \cdot) & \text{weakly in } L^\alpha(\Omega), \\
w^{\lambda_j}_\theta \rightharpoonup w_\theta &\text{weakly in } L^\alpha(B_R \times \Omega), \\
Dv^{\lambda_j}_\theta \rightharpoonup \Phi_\theta &\text{weakly in } L^\alpha(B_R \times \Omega; \mathbb{R}^n).
\end{array} \right.
\end{equation}
We can quickly show that $\overline{H}(p)$ is independent of $\omega$, i.e. $\overline{H}(p, \omega) = \overline{H}(p)$ for $\omega$ contained in a full probability set $\hat{\Omega}$. Due to the ergodicity hypothesis, it suffices to verify that for each $\mu \in \mathbb{R}$, the event $\{\omega \in \Omega: \overline{H}(p, \omega) \geq \mu\}$ is invariant under $\tau_y$, and this is true because $v^\lambda$ is stationary and (\ref{finaloscbound}).

Since $w^\lambda$ is continuous and satisfies (\ref{wlequation}), $w^\lambda_{\theta}$ converges to $w^\lambda$ locally uniformly as $\theta \rightarrow 0$, and $J$ has compact support, dominated convergence implies that as $\theta \rightarrow 0$,
\begin{equation}
\lambda (w^\lambda_{\theta}(z, \omega) + v^\lambda(0, \omega)) + \bar{J} - \int J(y)\exp(-y\cdot p) \exp(w^\lambda_{\theta}(z-y, \omega) - w^\lambda_{\theta}(z, \omega)) dy - c(z, \omega) = o_\theta(1). \label{theta}
\end{equation}
We also know that convex functionals are lower semicontinuous with respect to weak convergence, which follows from the fact that linear functionals are continuous with respect to weak convergence and that a convex functional is a supremum of linear functionals. Therefore, passing to the weak limit $\lambda_j \rightarrow 0$ in \eqref{theta} with $\theta$ fixed and applying (\ref{weakconvergence}) yields
\begin{equation}
\bar{J} - \int J(y) \exp(-y\cdot p) \exp(w_\theta(z-y, \omega) - w_\theta(z, \omega)) dy - c(z, \omega) \geq \overline{H}(p) - o_\theta(1) \label{wthetaequation}
\end{equation}
for $\omega \in \hat{\Omega}$. If we take a sufficiently small $\theta$ so that $o_\theta(1) < \nu$, $w_\theta$ is a function that satisfies (\ref{supercorrector}).

We now need to show \eqref{strictlysublinear}, i.e. that $w_\theta$ is strictly sublinear at infinity. First we note that $\Phi_\theta$ is stationary.
This is true because $v^\lambda$ is stationary, which means $v^\lambda_\theta := v^\lambda \ast \rho_\theta$ is also stationary. Therefore, $Dw^\lambda_\theta = Dv^\lambda_\theta$ is stationary, and so $\Phi_\theta$ is stationary as well. We can check (see \cite{kozlov}) that $\Phi_\theta = Dw_\theta$ almost surely in $\omega$, and that
$$
\mathbb{E}[\Phi_\theta(0, \cdot)] = \lim_{j \rightarrow \infty} \mathbb{E}[Dv^{\lambda_j}_\theta(0, \cdot)] = 0,
$$
which is true because $v^{\lambda_j}$ is stationary, and we have a local $L^\alpha$ bound on $Dw_\theta(\cdot, \omega)$. Thus because $w_\theta$ is a process with a stationary, mean zero gradient, it is strictly sublinear at infinity almost surely in $\omega$ (see \cite{kozlov}, \cite{takispaper}).

Our objective is now to show (\ref{fulllimsup}). First we show that
\begin{equation} \label{limsup}
-\overline{H} \geq \limsup_{\lambda \rightarrow 0} \lambda v^\lambda(0, \omega) \text{ a.s. in } \omega.
\end{equation}
To do this, we take a small $\nu$ to be fixed later and consider the approximate supercorrector $w_\nu$ satisfying (\ref{supercorrector}). Fix $\omega \in \tilde{\Omega} \subset \hat{\Omega}$, $\tilde{\Omega}$ a set of full probability where (\ref{strictlysublinear}) holds. Consider the function $\varphi(z) := \left(1+|z|^2\right)^{\frac{1}{2}}.$
For each $\lambda > 0$, define the function $\hat{w}^\lambda(z):= (1-\epsilon)(w_\nu(z, \omega) - (\overline{H} - \eta)\lambda^{-1}) + \epsilon \varphi(z),$ where $\eta, \epsilon > 0$ are small constants to be chosen. We would like to apply comparison to $v^\lambda$ and $\hat{w}^\lambda$. We use the convexity of the exponential and use the fact that $\varphi$ is uniformly Lipschitz continuous to conclude that
\begin{equation}
\label{whatlambda}
\lambda \hat{w}^\lambda + \bar{J} - \int J(y) \exp(-y\cdot p)\exp(\hat{w}^\lambda(z-y) - \hat{w}^\lambda(z))dy - c(z) \geq \lambda \hat{w}^\lambda + (1-\epsilon) \overline{H} - C\epsilon - \nu
\end{equation}
(\ref{strictlysublinear}) implies that for $R$ sufficiently large,
\begin{equation} \label{strictlysublinear2}
\inf_{B_R} w_\nu \geq C_\eta - \eta^3R
\end{equation}
We now choose $\epsilon = \min(\frac{1}{4}, \frac{\eta}{4C})$, with the same $C$ as in (\ref{whatlambda}), and so we can estimate the right hand side of (\ref{whatlambda}) in $B_R$ using (\ref{strictlysublinear2}) as follows:
$$
\lambda \hat{w}^\lambda + (1-\epsilon)\overline{H} + C\epsilon + o(\theta) = (1-\epsilon)(\lambda w_\nu - \eta) - C\epsilon - \nu \geq \lambda C_\eta - \lambda \eta^3 R + \frac{\eta}{2} - \nu.
$$
Now if we fix $\nu$ so that $\nu < \frac{\eta}{4}$, then for $\lambda$ sufficiently small this expression is nonnegative. Now we would like to verify that $\hat{w}^\lambda \geq v^\lambda$ on $\mathbb{R}^n \backslash B_R$. Because we have a uniform upper bound on $\lambda v^\lambda$, on $\mathbb{R}^n \backslash B_R$ we see that for some constant $C$, $\hat{w}^\lambda - v^\lambda \geq (1-\epsilon) w_\nu - C\lambda^{-1} + C\eta R.$ Therefore, because $w_\nu$ is strictly sublinear at infinity we can select $R:= C_1(\lambda\eta)^{-1}$ for a large constant $C_1$, and take $\lambda$ sufficiently small to conclude that $\hat{w}^\lambda \geq v^\lambda$ on $\mathbb{R}^n \backslash B_R$. We have obtained
\begin{equation*}
\left\{
\begin{array}{ll}
\displaystyle \lambda \hat{w}^\lambda + \bar{J} - \int J(y) \exp(-y\cdot p)\exp(\hat{w}^\lambda(z-y) - \hat{w}^\lambda(z))dy - c(z) \geq 0 & \text{in } B_R,\\
\hat{w}^\lambda \geq v^\lambda & \text{on }  \mathbb{R}^n \backslash B_R.\\
\end{array} \right.
\end{equation*}
By comparison (see \cite{nonlocaldirichlet}) we have that $\hat{w}^\lambda(\cdot) \geq v^\lambda(\cdot, \omega)$ in $B_R$ for $\lambda$ sufficiently small, and in particular we have that $\hat{w}^\lambda(0) \geq v^\lambda(0, \omega)$. Multiplying this inequality by $\lambda$, and taking $\lambda \rightarrow 0$, we have
$$
-\overline{H}+\eta \geq (1-\epsilon)^{-1} \limsup_{\lambda \rightarrow 0} \lambda v^\lambda(0, \omega),
$$
and because $\epsilon, \eta$ are arbitrarily small, and $\omega$ was an arbitrary element of $\tilde{\Omega}$, we have (\ref{limsup}). Since $-\overline{H}$ is the weak limit of the sequence $\lambda_j v^{\lambda_j}(0, \cdot)$, we have the reverse inequality of (\ref{limsup}) easily.
Therefore, we have shown (\ref{fulllimsup}).

It remains to prove (\ref{convergenceinEV}), and we first note that the full sequence $\lambda v^\lambda(0, \cdot)$ converges weakly to $-\overline{H}$ in $L^\alpha$, in particular in $L^1$.
Using a measure-theoretic lemma (see \cite{takispaper}, \cite{lionssouganidis}), we have
\begin{equation} \label{blah3}
\mathbb{E}[|\lambda v^\lambda(0, \cdot) + \overline{H}|]\rightarrow 0 \text{ as } \lambda \rightarrow 0.
\end{equation}
Fix $R > 0$, let $\gamma > 0$ be a small constant, and use the Vitali covering lemma to select points $z_1, \ldots, z_k \in B_R$ such that
$$
B_R \subset \bigcup_{i=1}^k B(z_i, \gamma) \text{ and } k \leq C\left(\frac{R}{\gamma}\right)^d.
$$
Now using (\ref{blah3}), (\ref{finaloscbound}), the fact that $v^\lambda$ is stationary, and the fact that expectation is preserved under measure-preserving transformations, we have
\begin{align*}
\limsup_{\lambda \rightarrow 0}\quad &\mathbb{E}\left[\sup_{z \in B_{\frac{R}{\lambda}}} |\lambda v^\lambda(z, \omega) + \overline{H}|\right] \\
&\leq \max_{1 \leq i \leq k}\left[\limsup_{\lambda \rightarrow 0} \mathbb{E}\left[|\lambda v^\lambda\left(\frac{z_i}{\lambda}, \omega\right) + \overline{H}|\right]\right] + \limsup_{\lambda \rightarrow 0} \mathbb{E}\left[\max_{1 \leq i \leq k} \osc_{z \in B(\frac{z_i}{\lambda}, \frac{\gamma}{\lambda})} \lambda v^\lambda(z, \omega)\right] \\
&\leq \max_{1 \leq i \leq k} \left[ \limsup_{\lambda \rightarrow 0} \mathbb{E}\left[|\lambda v^\lambda(0, \omega) + \overline{H}|\right]\right] + C\gamma  = C \gamma,
\end{align*}
Because $\gamma$ is an arbitrary small constant, we get (\ref{convergenceinEV}).
\end{proof}
We next show that $\overline{H}(p)$ is concave, negatively coercive, and continuous in $p$ a.s. in $\omega$. Because the proof is the same as the proof of Proposition 4.6 in \cite{almostperiodicpaper}, we omit it.
\begin{proposition} \label{Hbarproperties}
The effective Hamiltonian $\overline{H}$ satisfies the following properties a.s. in $\omega$:
\begin{compactenum}
\item $p \mapsto \overline{H}(p)$ is concave.
\item There exist constants $K_1, K_2, K_3, K_4> 0, C_1, C_2$ such that $\overline{H}(p)$ satisfies
\begin{equation} \label{hbarcoercive}
-K_1 \exp(K_2|p|) - C_1 \geq \overline{H}(p) \geq -K_3 \exp(K_4|p|) - C_2
\end{equation}
for all $p \in \mathbb{R}^n$. In particular, this implies that $\overline{H}$ is uniformly and negatively coercive.
\item There exist constants $C_3, C_4$ such that for all $p_1, p_2 \in \mathbb{R}^n$,
\begin{equation} \label{hbarloclipschitz}
|\overline{H}(p_1) - \overline{H}(p_2)| \leq C_3\exp(C_4(1+|p_1|+|p_2|))|p_1 - p_2|.
\end{equation}
\end{compactenum}
\end{proposition}
\begin{remark} \label{convsubsequence}
(\ref{convergenceinEV}) implies that for a fixed $p$ there exists a subsequence $\lambda_j \rightarrow 0$, a full probability set $\Omega_2 \subset \Omega_1$ such that for every $R > 0$,
\begin{equation} \label{convsubsequence1}
\lim_{j \rightarrow \infty} \sup_{z \in B_{\frac{R}{\lambda_j}}} |\lambda_j v^{\lambda_j}(z, \omega; p) + \overline{H}(p)| = 0.
\end{equation}
Via a diagonalization procedure we can obtain a subsequence such that (\ref{convsubsequence1}) holds for every $R > 0$, rational $p$, and $\omega \in \Omega_2$, and then use Lemma \ref{plipschitzlemma} to conclude that (\ref{convsubsequence1}) holds for all $p \in \mathbb{R}^n, \omega \in \Omega_2$.
\end{remark}
\begin{remark} \label{liminfremark}
Proposition \ref{firstconvergenceprop} has given us a partial homogenization result, in the direction of \eqref{asconvergence}, which represents the ultimate goal of proving almost sure homogenization. Given (\ref{fulllimsup}), to finish the proof of \eqref{asconvergence} it remains to show that $\displaystyle \liminf_{\lambda \rightarrow 0} \lambda v^\lambda(0, \omega; p) = -\overline{H}(p).$ The methods of Section \ref{sec:metricproblem} are required to prove this. However, we note that the quantity
\begin{equation} \label{hhatfullprob}
-\hat{H}(p, \omega) := \displaystyle \liminf_{\lambda \rightarrow 0} \lambda v^\lambda(0, \omega; p)
\end{equation}
is deterministic almost surely in $\omega$ because the set
$
\{\omega: -\hat{H}(p, \omega) \geq \mu\}
$
is invariant under $\tau_y$ for every $y \in \mathbb{R}^d$.
\end{remark}
\section{Metric Problem}
\label{sec:metricproblem}
\subsection{Well-Posedness of the Metric Problem}
We define the metric problem to be the equation
\begin{equation} \label{metricequation}
\bar{J}-\int J(y) \exp(-y\cdot p) \exp(m_\mu(z-y, \omega) - m_\mu(z, \omega))dy - c(z, \omega) = \mu \text{ on } \mathbb{R}^n \backslash D
\end{equation}
where $D$ is a bounded closed set and $\mu \in \mathbb{R}$. Subject to appropriate boundary conditions for solutions on $\mathrm{int}(D)$ and at infinity, a solution to (\ref{metricequation}) is related to the ``metric'' associated with the effective Hamiltonian. This construction is the direct analogue of the metric problem from \cite{takispaper}. Our objective in this subsection is to show that the metric problem is well posed for each $\mu < \overline{H}(p)$, and the ``approximate supercorrectors'' from Section \ref{sec:effhamiltonian} serve as a crucial part of the proof. In this entire section let $\nu$ be sufficiently small so that $\overline{H}(p) - \nu > \mu$, and let $w_\nu$ satisfy (\ref{supercorrector}) with this value of $\nu$. We first prove a comparison principle for (\ref{metricequation}).

\begin{proposition} \label{comparison}
Fix $p \in \mathbb{R}^n$, $\mu < \overline{H}(p)$, and $\omega \in \Omega_3$. Let $D \subset \mathbb{R}^d$ be closed and bounded. Suppose that $u \in \mathrm{USC}(\mathbb{R}^n), v \in \mathrm{LSC}(\mathbb{R}^n)$ are respectively a subsolution and supersolution of \eqref{metricequation} and satisfy $u \leq v$ on $\mathrm{int}(D)$, and we have the following conditions at infinity:
\begin{equation} \label{metricprobleminftycondition}
\limsup_{|z| \rightarrow \infty} \frac{u(z)}{|z|} \leq 0, \liminf_{|z| \rightarrow \infty} \frac{u(z)}{|z|} > -\infty, \liminf_{|z|\rightarrow \infty} \frac{v(z)}{|z|} > -\infty.
\end{equation}
Then $u \leq v$ on $\mathbb{R}^n\backslash D$.
\end{proposition}
Note that it is sufficient to assume that $u \leq v$ in the interior of $D$, rather than on $D$; this is a departure from the typical assumption for comparison in a nonlocal Dirichlet problem, which requires the assumption that $u \leq v$ on the complement of the domain, which in this case is $D$. This relaxed assumption is crucial in the subsequent Perron construction to ensure well-posedness.
\begin{proof}
We will omit dependence on $\omega$ and assume that $p = 0$ for simplicity. Adjust $w_\nu$ by a constant so that $w_\nu \geq u$ on $D$. We claim that it suffices to consider $v$ that satisfies
\begin{equation} \label{newinfinitycondition}
\liminf_{|z| \rightarrow \infty} \frac{v(z)}{|z|} \leq 0.
\end{equation}
Suppose that instead the opposite were true. That is,
\begin{equation} \label{oppositeinfinitycondition}
\liminf_{|z| \rightarrow \infty} \frac{v(z)}{|z|} > 0.
\end{equation}
In this case, we define $\tilde{v} := (1-\epsilon)v + \epsilon w_\nu.$ Convexity implies that $\tilde{v}$ is a strict supersolution of (\ref{metricequation}); that is,
\begin{equation} \label{strictsupersolution}
\bar{J}-\int J(y) \exp(-y\cdot p) \exp(\tilde{v}(z-y, \omega) - \tilde{v}(z, \omega))dy - c(z, \omega) \geq \mu + \epsilon_2
\end{equation}
for some $\epsilon_2 > 0$. Because $w_\nu \geq u$ on $D$, $\tilde{v} \geq u$ on $\mathrm{int}(D)$. We now prove a comparison result for (\ref{metricequation}) between strict supersolutions and subsolutions with the assumption that the subsolution lies below the supersolution on the interior of $D$. For use later, we will also state this comparison result in the case that (\ref{metricequation}) is considered in a bounded domain, but because the kernel $J$ is compactly supported, in that case we only need to assume $u \leq v$ on a compact subset of the unbounded set $O^c$.
\begin{lemma} \label{comparisonlemma}
Suppose that in an open set $O \subset \mathbb{R}^n$, $v$ satisfies \eqref{strictsupersolution} for some $\epsilon_2 > 0$ i.e. is a strict supersolution of \eqref{metricequation}, and $u$ is a subsolution of \eqref{metricequation}. We also assume that $O$ is either bounded or has bounded complement. If $O = \mathbb{R}^n \backslash D$ for a bounded, closed $D$, assume that $u \leq v$ on $\mathrm{int}(D)$, \eqref{metricprobleminftycondition}, and \eqref{oppositeinfinitycondition}. If $O$ is bounded, assume that $u \leq v$ on
$$
E := \{z + z_1: z \in O, z_1 \in B(0, \overline{r})\} \cap \mathrm{int}(O^c).
$$
Then $u \leq v$ on $O$.
\end{lemma}
\begin{proof}
We consider the case where $O = \mathbb{R}^n \backslash D$ and $D$ is a closed and bounded set. Define $M = \sup_{z \in \mathbb{R}^n \backslash \partial D} [u(z) - v(z)].$ Suppose for a contradiction that $M > 0$. Then there exists a large ball $O_2$ such that for any $\gamma > 0$ there is a point $z_0 \in O_2 \backslash D$ satisfying $u(z_0) - v(z_0) > M - \gamma$, due to (\ref{oppositeinfinitycondition}), (\ref{metricprobleminftycondition}), and the fact that $u \leq v$ on $\mathrm{int}(D)$. This means that for any $z \in \mathbb{R}^n \backslash \partial D$, we have that $u(z) - v(z) \leq u(z_0) - v(z_0) + \gamma$, which means that
\begin{equation} \label{blah8}
u(z) - u(z_0) \leq v(z) - v(z_0) + \gamma \text{ for } z \in \mathbb{R}^n \backslash \partial D.
\end{equation}
However, we know that
\begin{align} \label{blah1}
\begin{split}
\bar{J} - \int J(y) \exp(-y\cdot p)\exp(u(z_0 - y) - u(z_0)) dy - c(z_0) &\leq \mu \\
\bar{J} - \int J(y) \exp(-y\cdot p)\exp(v(z_0 - y) - v(z_0)) dy - c(z_0) &\geq \mu + \epsilon_2.
\end{split}
\end{align}
and $\partial D$ is of measure zero, so it can be excluded from the domain of integration. This means that at $z = z_0$, upon applying (\ref{blah8}) to the first equation of \eqref{blah1} and subtracting it from the second, we get
\begin{equation} \label{newcomparisonthing}
(\exp(\gamma)-1)\int J(y) \exp(-y\cdot p)\exp(v(x_0 - y) - v(x_0)) dy \geq \epsilon_2.
\end{equation}
However, since $v$ is a supersolution of (\ref{metricequation}), we know that
$$
\displaystyle \int J(y) \exp(-y\cdot p)\exp(v(z_0 - y) - v(z_0)) dy \leq C,
$$
where $C$ is uniform in $z_0$. This means that we can take $\gamma$ sufficiently small in (\ref{newcomparisonthing}) to reach a contradiction.

If $O$ is a bounded set, then the proof proceeds mostly in the same manner
In this case we define $M = \sup_{z \in O} [u(z) - v(z)],$ and $O$ takes the place of $O_2$ in the above proof. Then \eqref{blah8} holds for $z \in E$, and because the support of $J$ is contained in $B(0, \overline{r})$, we know that \eqref{blah8} holds for every point in the domain of integration. Therefore, we can apply \eqref{blah8} and proceed with the rest of the proof to reach the same conclusion.
\end{proof}
Lemma \ref{comparisonlemma} implies that $u \leq \tilde{v}$, which implies that $u \leq v$ upon taking $\epsilon \rightarrow 0$. Therefore, if we assume (\ref{oppositeinfinitycondition}) then we have Proposition \ref{comparison}. So it suffices to consider $v$ that satisfies (\ref{newinfinitycondition}). Define
$$
\Lambda := \left\{0 \leq \lambda \leq 1: \liminf_{|z|\rightarrow \infty} \frac{\lambda v(z) - u(z)}{|z|} \geq 0\right\} \text{ and } \overline{\lambda} := \sup{\Lambda}.
$$
We first show that $\Lambda = [0, \overline{\lambda}]$. (\ref{metricprobleminftycondition}) implies that $0 \in \Lambda$, (\ref{newinfinitycondition}) implies that if $a, b \in \Lambda$, then $[a, b] \subset \Lambda$, and we can check that $\overline{\lambda} \in \Lambda$.
Our objective is now to show that $\overline{\lambda} = 1$. We suppose that $\overline{\lambda} > 0$; this can be proven by an argument very similar to the one that follows.
For each $R > 1$, define the auxiliary function $\varphi_R(z) := R - (R^2 + |z|^2)^{\frac{1}{2}}.$
Note that the first and second derivatives of this function are bounded independent of $R$, $-\varphi_R$ grows linearly at infinity, and $\varphi_R \rightarrow 0$ as $R \rightarrow \infty$.
Take $\lambda$ satisfying $0 < \lambda < \overline{\lambda} \leq 1$. Fix constants $0 < \eta < 1$ and $\theta > 1$ to be selected subsequently. Define
$$
\hat{v} := \lambda v + (1-\lambda) w_\nu, u_{R} := (1-\eta)u + \eta \theta \varphi_R.
$$
Using \eqref{metricprobleminftycondition} and the fact that $w_\nu$ is strictly sublinear at infinity, we can take
$$
\theta:= -2\liminf_{|z| \rightarrow \infty} \frac{u(z)}{|z|},
$$
so that for any $\eta > 0, R > 1$,
$$
\liminf_{|z|\rightarrow \infty} \frac{\hat{v}(z) - u_R(z)}{|z|} = \displaystyle \liminf_{|z| \rightarrow \infty} \frac{\lambda v(z) - u(z)}{|z|} + \frac{\eta u(z) - \eta \theta \varphi_R(z)}{|z|} > 0.
$$
Because the nonlocal term is convex and $\lambda < 1$, we have that
$$
\bar{J}-\int J(y) \exp(\hat{v}(z-y) - \hat{v}(z))dy - c(z) \geq \lambda \mu + (1-\lambda)(\overline{H}(0)-\nu) =: H > \mu.
$$
In addition, since $\varphi_R$ has bounded derivatives, we can take $\eta$ sufficiently small depending on $\theta$ so that
$$
\bar{J} - \int J(y) \exp(u_R(z-y) - u_R(z)) dy - c(z) < H \text{ on } \mathbb{R}^n.
$$
Applying Lemma \ref{comparisonlemma} to $\hat{v}$ and $u_R$, sending $R \rightarrow \infty$, and noting that $w_\nu \geq u$ on $D$, we have that
\begin{equation}
\label{uvhatinequality}
\hat{v} - (1-\eta)u \geq \inf_{\mathrm{int}(D)} (\hat{v} - (1-\eta)u) \geq 0 \text{ in } \mathbb{R}^n \backslash D.
\end{equation}
Since $w_\nu$ is strictly sublinear at infinity, it follows that
$$
\liminf_{|z|\rightarrow \infty} \frac{\lambda v(z) - (1-\eta)u(z)}{|z|} \geq 0,
$$
and therefore $\overline{\lambda} \geq \frac{\lambda}{1-\eta}$. If $\overline{\lambda} < 1$, then we can take $\lambda \rightarrow \overline{\lambda}$ to obtain $\overline{\lambda} \geq \frac{\overline{\lambda}}{1-\eta}$, which is a contradiction. Therefore, $\overline{\lambda} = 1$. This implies that the preceding analysis holds for any $0 < \lambda < 1$, and so taking $\eta \rightarrow 0$ and $\lambda \rightarrow 1$ in (\ref{uvhatinequality}) shows that $u \leq v$, which was what we wanted to show.
\end{proof}
Now that we have proven a comparison principle, we will use Perron's method to prove a kind of well-posedness statement for (\ref{metricequation}). However, because nonlocal Dirichlet problems require the boundary condition to be prescribed on the entire complement of the domain, rather than just the boundary, our barrier will be discontinuous on the boundary. Fortunately, Lemma \ref{comparisonlemma} allows us to apply comparison in this situation.
\begin{proposition} \label{metricproblemwellposed}
For each fixed $p, z_1 \in \mathbb{R}^n$, $\omega \in \Omega_3$, and $\mu < \overline{H}(p)$, there exists a solution $m_\mu(z) = m_\mu(z, z_1, \omega; p) \in C(\mathbb{R}^n \backslash \partial D_1)$ of \eqref{metricequation}, a.s. in $\omega$, where $D_1(z_1) = z_1 + B_1$, and $m_\mu$ satisfies the boundary condition
\begin{equation}
\label{boundarycondition}
m_\mu(z, z_1, \omega) = w_\nu(z, \omega) - w_\nu(z_1, \omega) \text{ on } D_1(z_1)
\end{equation}
and at infinity satisfies
\begin{equation} \label{inftycondition2}
\limsup_{|z| \rightarrow \infty} \frac{m_\mu(z, z_1, \omega)}{|z|} \leq 0.
\end{equation}
This solution is uniquely defined except on $\partial D_1$. In addition, the discontinuity of $m_\mu$ at $\partial D_1$ is a jump discontinuity and the size of the jump is bounded by a constant that depends on $\mu$ and the behavior of $w_\nu(z) - w_\nu(z_1)$ on $D_1(z_1)$.
\end{proposition}
\begin{proof}
Assume without loss of generality that $p = 0, z_1 = 0$, and suppose for simplicity $\mathrm{supp}(J)$ is contained in the unit ball; if the support of $J$ is larger the proof can be adjusted by simply taking $D_1$ to be a larger ball. Also suppose for simplicity that $J$ is constant; this simplification can be made because $J(0) > 0$ and $J$ is symmetric.
We define
\begin{multline}
m_\mu(z) := \sup\{u(z) : u \in \mathrm{USC}(\mathbb{R}^n), u \text{ is a subsolution of } (\ref{metricequation}) \text{ in } \mathbb{R}^n  \backslash \overline{D}_1, \\ u(z) \leq w_\nu(z) - w_\nu(0) \text{ on } D_1, \text{ and } \limsup_{|z|\rightarrow \infty} \frac{u(z)}{|z|} \leq 0\}
\end{multline}
To show that $m_\mu$ exists, we need to show that the admissible set is nonempty, that is, we must build a barrier. We can check that for $a_1 > 0$ sufficiently large and a suitable constant $a_2$, $\tilde{u}(z) := -a_1|z|-a_2$ is an admissible subsolution of (\ref{metricequation}).
Define $A_z := \{y \in B(z, 1): |y| \leq |z|-\frac{1}{2}\}.$ Because $z \notin B(0, 1)$, $|A_z| \geq C > 0$ uniformly in $z$. This means that
\begin{align*}
\int_{B(0, 1)} \exp(a_1 (|z| - |z-y|))dy &\geq \int_{A_z} \exp(a_1(|z| - |y|))dy \\
&\geq \int_{A_z} \exp\left(\frac{a_1}{2}\right) dy = C \exp\left(\frac{a_1}{2}\right),
\end{align*}
In addition, if we take $a_2 = -\inf_{z \in \overline{D}_1} [w_\nu(z) - w_\nu(0)],$ then $\tilde{u}(z) \leq w_\nu(z) - w_\nu(0)$ on $D_1$, so $\tilde{u}(z)$ is an admissible subsolution of \eqref{metricequation} for $a_1$ large and $a_2$ a suitable constant. This implies that $m_\mu$ is well defined. On $\mathbb{R}^n \backslash \overline{D}_1$, we have that $(m_\mu)^*$ is a subsolution of (\ref{metricequation}), because it is the supremum of subsolutions, and the Perron construction (see \cite{nonlocalcomparison} for the construction in the nonlocal setting) gives us that $(m_\mu)_*$ is a supersolution of (\ref{metricequation}). In addition, note that $m_\mu(z) = w_\nu(z) - w_\nu(0)$ on $D_1$ because any subsolution of \eqref{metricequation} on $\mathbb{R}^n \backslash \overline{D}_1$ remains a subsolution if its value is increased on $D_1$. It is true by construction that $m_\mu(z) \leq w_\nu(z) - w_\nu(0)$, and any subsolution $u$ of (\ref{metricequation}) on $\mathbb{R}^n \backslash \overline{D}_1$ remains a subsolution on $\mathbb{R}^n \backslash \overline{D}_1$ if $u$ is increased on $D_1$.

Therefore, $m_\mu$ is a solution of (\ref{metricequation}) satisfying (\ref{boundarycondition}). Because $w_\nu$ is continuous, we have that $(m_\mu)^*(z) = (m_\mu)_*(z)$ on $D_1$. This means that we can apply Proposition \ref{comparison} to say that
$m_\mu$ is uniquely defined and continuous on $\mathbb{R}^n \backslash \partial D_1$. Note, however, that $\tilde{u}(z)$ is not equal to $w_\nu(z) - w_\nu(0)$ on $\partial D_1$. This is also not merely a matter of adjusting $\tilde{u}(z)$ at the boundary, because the barrier must lie below $w_\nu(z) - w_\nu(0)$ on all of $D_1$, rather than just at the boundary, and due to the nonlocal character of the equation, changing the values of $\tilde{u}(z)$ on $D_1$ can affect whether it is a subsolution of \eqref{metricequation}. This causes the resulting solution $m_\mu$ to have a possible discontinuity at $\partial D_1$.
We can check that the size of the jump discontinuity of $m_\mu$ at $\partial D_1$ is bounded above by $\sup_{z \in \partial D_1} [w_\nu(z) - w_\nu(0)] + a_1 + a_2$, and so we have the claim of Proposition \ref{metricproblemwellposed}.
\end{proof}

\begin{remark}
Because the value of $m_\mu$ on $\partial D_1$ is not unique, we will define the value of $m_\mu$ there to be $m_\mu(z, z_1, \omega) = w_\nu(z, \omega) - w_\nu(z_1, \omega)$. This implies that $m_\mu(z, z_1, \omega) = w_\nu(z, \omega) - w_\nu(z_1, \omega)$ on $D_1$, an essential property for showing that $m_\mu$ is superadditive and jointly stationary.
\end{remark}

\subsection{Properties of $m_\mu$}
In this subsection, we prove various properties of $m_\mu$. We first show that $m_\mu$ satisfies the same oscillation bounds as those proven in Section \ref{sec:estimates}.
\begin{proposition} \label{metricproblemoscprop}
Assume that $\mu < \overline{H}(p)$ and that the assumptions of Proposition \ref{osctheorem} hold. There exists a uniform constant $C(p, \mu)$ such that for all $z, z_1 \in \mathbb{R}^n$, $R > \frac{1}{2}$ and $\omega \in \Omega_3$,
\begin{equation} \label{metricproblemoscboundfirst}
\osc_{B(z, R)} m_\mu(\cdot, z_1, \omega) \leq CR.
\end{equation}
In particular, for any $r > 0$ there exists a uniform constant $C$ such that for all $z, z_1 \in \mathbb{R}^n$
\begin{equation} \label{metricproblemoscbound}
\limsup_{\lambda \rightarrow 0} \osc_{B(\frac{z}{\lambda}, \frac{r}{\lambda})} \lambda m_\mu(\cdot, z_1, \omega) \leq Cr.
\end{equation}
\end{proposition}
\begin{proof}
We first state a lemma showing that $m_\mu$ must have linear growth at infinity. Because the proof follows in the same way as Remark 6.4 of \cite{takispaper}, we will omit it.
\begin{lemma} \label{metricfunctionlemma}
For each $\mu < \overline{H}(p)$ and $\omega \in \Omega_3$, $\displaystyle \limsup_{|z| \rightarrow \infty} \frac{m_\mu(z, 0, \omega)}{|y|} < 0.$
\end{lemma}
We now move to Proposition \ref{metricproblemoscprop}. Because the proof will involve repeating the arguments of Section \ref{sec:estimates}, we will give a sketch of the proof. First we note that $w_\nu$ is Lipschitz continuous, and the jump discontinuity at $\partial B(0, 1)$ is controlled by Proposition \ref{metricproblemwellposed}. Therefore, since we are considering $R > \frac{1}{2}$, we can essentially consider $m_\mu$ and \eqref{metricequation} as if it were set in $\mathbb{R}^n$ by increasing $C$.
We would like to apply the methods of Section \ref{sec:estimates}. Define
$$
\Phi(z) := \int J(y)\exp(-y\cdot p) \exp(m_\mu(z-y, z_1) - m_\mu(z, z_1)) dy,
$$
and note that this is the nonlocal term of \eqref{metricequation} and is the same as the definition of $\Phi^\lambda$ with $m_\mu$ in the place of $w^\lambda$. We now need to prove that $\Phi$ is uniformly bounded above and uniformly positive, as these are the essential properties of $\Phi$ required to apply the proof of Proposition \ref{psiboundprop} and Proposition \ref{osctheorem}.

We can show that $\Phi$ must be uniformly (in $z$) bounded above by noting that $m_\mu$ solves (\ref{metricequation}), that $c(z, \omega)$ is uniformly bounded, and that $\mu$ is a constant. To show that $\Phi$ is uniformly positive, we will make use of Lemma \ref{metricfunctionlemma} and \eqref{c0assumption}. We claim that $\Phi(z) \geq \bar{J} - \rho > 0$ for all $z \in \mathbb{R}^n$. Suppose for a contradiction that there is a point $\hat{z}$ such that $\Phi(\hat{z}) = \eta < \bar{J} - \rho$. Then we can write $\mu = \bar{J} - c(\hat{z}) - \eta$. If we now consider a constant function $\varphi(z) := C_1$, and insert it into \eqref{metricequation} in the place of $m_\mu$, then on the left hand side for any $z \in \mathbb{R}^n$ we get
\begin{equation*}
\bar{J}-\int J(y) \exp(-y\cdot p) dy - c(z) \leq -c(z) \leq \bar{J} - c(\hat{z}) - \eta = \mu,
\end{equation*}
where the second inequality follows due to \eqref{c0assumption} and our assumption that $\eta < \bar{J} - \rho$. This means that $\varphi$ is a subsolution of \eqref{metricequation} on $\mathbb{R}^n$, and upon taking $C_1 := \inf_{D_1} w_\nu(\cdot) - w_\nu(z_1)$, we can apply Proposition \ref{comparison} to conclude that $C_1 \leq m_\mu$ on $\mathbb{R}^n$, which contradicts Lemma \ref{metricfunctionlemma}.

Therefore, if we define $\tilde{m}_\mu(\cdot, z_1) := m_\mu(\cdot, z_1) - m_\mu(z, z_1)$,  then we can apply the overlapping ball arguments of Propositions \ref{osctheorem}, \ref{lptheorem}, and \ref{psiboundprop} to conclude that \eqref{metricproblemoscboundfirst} holds for $\tilde{m}_\mu$, which of course means that it holds for $m_\mu$. The same constant will work for every $z \in \mathbb{R}^n$, and the fact that $w_\nu$ is uniformly Lipschitz continuous along with the fact that the barrier construction $\tilde{u}$ from Proposition \ref{metricproblemwellposed} works for any $z_1$ via simple translation implies that the constant is uniform in $z_1$ as well.
\end{proof}
In addition, by comparison and the stationarity of the coefficients of (\ref{metricequation}), we can show that $m_\mu$ is jointly stationary.
\begin{proposition} \label{jointlystationary}
The functions $m_\mu$ are jointly stationary i.e. for any $z, z_1, z_2 \in \mathbb{R}^n$:
$$
m_\mu(z, z_1, \tau_{z_2} \omega) = m_\mu(z + z_2, z_1 + z_2, \omega).
$$
\end{proposition}
Next we prove a superadditivity property of $m_\mu$, which will be critical in applying the subadditive Ergodic Theorem. We observe that due to Propositions \ref{finaloscbound} and \ref{metricproblemoscprop}, there exists a uniform (in $z, z_1$) constant $C$ such that
\begin{equation} \label{wandmbound}
\sup_{z_2 \in B(z, 1)} (|w_\nu(z, \omega) - w_\nu(z_2, \omega)| + |m_\mu(z_2, z_1, \omega) - m_\mu(z, z_1, \omega)|) \leq C \text{ a.s. in } \omega.
\end{equation}
Now we define $\hat{m}_\mu(z, z_1, \omega):= m_\mu(z, z_1, \omega) - C,$ and we claim that $\hat{m}_\mu$ is superadditive.
\begin{proposition} \label{superadditivity}
For each $p \in \mathbb{R}^d, \mu < \overline{H}(p)$, $\omega \in \Omega_3$, and for all $z, z_1, z_2 \in \mathbb{R}^n$,
\begin{equation} \label{subadd}
\hat{m}_\mu(z, z_1, \omega) \geq \hat{m}_\mu(z, z_2, \omega) + \hat{m}_\mu(z_2, z_1, \omega).
\end{equation}
\end{proposition}
\begin{proof}
We fix $z_1, z_2$ and rewrite (\ref{subadd}) in the form
\begin{equation} \label{subadd2}
m_\mu(z, z_1, \omega) \geq m_\mu(z, z_2, \omega) + m_\mu(z_2, z_1, \omega) - C.
\end{equation}
Considering both sides of (\ref{subadd2}) as functions of $z$ and taking $D = B(z_1, 1) \cup B(z_2, 1)$, we see that it suffices by Proposition \ref{comparison} to show that (\ref{subadd2}) holds for $z \in\overline{B(z_1, 1) \cup B(z_2, 1)}$.
We have by Proposition \ref{comparison} that for all $z \in \mathbb{R}^n$,
\begin{equation} \label{supaddcomparisonthing}
w_\nu(z, \omega) - w_\nu(z_2, \omega) \geq m_\mu(z, z_2, \omega).
\end{equation}
Now we use (\ref{supaddcomparisonthing}) and the fact that $m_\mu(z, z_1) = w_\nu(z) - w_\nu(z_1)$ for all $z \in \overline{B(z_1, 1)}$, to obtain that for all $z \in \overline{B(z_1, 1)}$,
\begin{align*}
m_\mu(z, z_1, \omega) &= w_\nu(z, \omega) - w_\nu(z_1, \omega) \\
&= (w_\nu(z, \omega) - w_\nu(z_2, \omega)) - (w_\nu(z_1, \omega) - w_\nu(z_2, \omega))\\
&\geq m_\mu(z, z_2, \omega) + m_\mu(z_2, z_1, \omega).
\end{align*}
Then, for $z \in \overline{B(z_2, 1)}$, we use (\ref{wandmbound}) to get
\begin{align*}
m_\mu(z, z_1, \omega) &\geq w_\nu(z, \omega) - w_\nu(z_2, \omega) - C + m_\mu(z_2, z_1, \omega) \\
&= m_\mu(z, z_2, \omega) + m_\mu(z_2, z_1, \omega) - C.
\end{align*}
This finishes the proof.
\end{proof}
We also note that $m_\mu$ is Lipschitz continuous with respect to its second parameter. The proof can be found in \cite{ben}.
\begin{lemma} \label{lemma1}
For any $\omega \in \Omega_3$, there exists a uniform (in $z$) constant $C$ such that
\begin{equation} \label{secondvarlipschitz}
|m_\mu(z, z_1, \omega) - m_\mu(z, z_2, \omega)| \leq C|z_1 - z_2|.
\end{equation}
\end{lemma}
\subsection{Homogenization of the Metric Problem}
We now consider the boundary value problem
\begin{equation} \label{scaledmetricproblem}
\left\{
\begin{array}{l}
\displaystyle \bar{J}-\int \epsilon^{n} J_\epsilon(y) \exp\left(-y \cdot p + \frac{m_\mu^\epsilon(z - y, \omega) -  m^\epsilon_\mu(z, \omega)}{\epsilon}\right)dy - c\left(\frac{z}{\epsilon}, \omega\right) = \mu \text{ in } \mathbb{R}^n \backslash D_\epsilon,\\
\displaystyle m^\epsilon_\mu(z, \omega) = \epsilon w_\nu\left(\frac{z}{\epsilon}\right) - \epsilon w_\nu(0) \text{ on } D_\epsilon, \\
 \displaystyle \limsup_{|z|\rightarrow \infty} |z|^{-1} m^\epsilon_\mu(z, \omega)\leq 0 \text{ a.s. in } \omega,
\end{array} \right.
\end{equation}
where $J_\epsilon(y) := J(\frac{y}{\epsilon})$ and $D_\epsilon = B(0, \epsilon)$. This problem is a rescaling of (\ref{metricequation})-(\ref{boundarycondition})-(\ref{inftycondition2}), and it is easy to see that a solution of this equation is $m^\epsilon_\mu(z, \omega) = \epsilon m_\mu\left(\frac{z}{\epsilon}, 0, \omega\right).$ We show that as $\epsilon \rightarrow 0$, $m^\epsilon_\mu$ converges almost surely to the unique solution $\overline{m}_\mu(y; p)$ of
\begin{equation} \label{effhamiltonian}
\left\{
\begin{array}{l}
\displaystyle \overline{H}(p + D\overline{m}_\mu) = \mu \text{ in } \mathbb{R}^n \backslash \{0\}, \\
\displaystyle \overline{m}_\mu(0) = 0, \\
\displaystyle \limsup_{|z|\rightarrow \infty} |z|^{-1} \overline{m}_\mu(z) \leq 0.
\end{array} \right.
\end{equation}
(\ref{effhamiltonian}) is well-posed in $C(\mathbb{R}^n)$ for every $\mu < \overline{H}(p)$ by arguments that are similar to the preceding subsections. The non-uniqueness at the boundary discussed by Proposition \ref{metricproblemwellposed} is not present because for \eqref{effhamiltonian} the boundary is just the origin, rather than a ball.
\begin{proposition} \label{metricproblemhomogthm}
There exists a full probability subset $\Omega_4 \subset \Omega_3$ such that for each $p, y \in \mathbb{R}^n$, $\mu < \overline{H}(p)$, and $\omega \in \Omega_4$,
\begin{equation} \label{metricproblemconvergence}
\lim_{t \rightarrow \infty} \frac{1}{t} m_\mu(tz, 0, \omega; p) = \overline{m}_\mu(z; p).
\end{equation}
\end{proposition}
\begin{proof}
We know by comparison that $m_\mu$ and $\hat{m}_\mu$ differ by no more than a constant. Therefore, because $\hat{m}_\mu$ is superadditive, we can apply the subadditive ergodic theorem using the semigroup $\sigma_t := \tau_{tz}$ to say that for each $z \in \mathbb{R}^d$, as $t \rightarrow \infty$, $t^{-1}m_\mu(tz, 0, \omega) = t^{-1} \hat{m}_\mu(tz, 0, \omega) \rightarrow M_\mu(z, \omega),$ for some $M_\mu(z, \omega)$ and $\omega \in \Omega_z$, $\Omega_z$ a set of full probability that depends on $z$. Define $\Omega_4 = \bigcup_{z \in \mathbb{Q}^n} (\Omega_z \cap \Omega_3).$ Then $\Omega_4$ is also a set of full probability, and we have that for $z \in \mathbb{Q}^n, \omega \in \Omega_4$,
\begin{equation} \label{convergenceonrationals}
t^{-1}m_\mu(tz, 0, \omega) \rightarrow M_\mu(z, \omega).
\end{equation}
We would like to show that \eqref{convergenceonrationals} holds for all $z \in \mathbb{R}^n$. Take a sequence $z_i \in \mathbb{Q}^n$ such that $z_i \rightarrow z$. Using (\ref{metricproblemoscbound}), we can conclude that for $\omega \in \Omega_4$, as $t \rightarrow \infty$,
\begin{equation} \label{lipschitzresult}
t^{-1} (m_\mu(tz, 0, \omega) - m_\mu(tz_i, 0, \omega)) \leq C|z - z_i|.
\end{equation}
Now applying (\ref{lipschitzresult}) with $z_j$ in place of $z$, we have that $t^{-1} m_\mu(tz_i, 0, \omega)$ is a Cauchy sequence, which means that $M_\mu(\cdot, \omega)$ can be defined on $\mathbb{R}^n$ as a continuous function.
This means that for $\omega \in \Omega_4$ and any $z \in \mathbb{R}^n$, $\lim_{t \rightarrow \infty} t^{-1} m_\mu(tz, 0, \omega) = \lim_{i \rightarrow \infty} M_\mu(z_i, \omega) = M_\mu(z, \omega).$ It remains to show that $M_\mu(z, \omega) = \overline{m}_\mu(z)$. First we check that $M_\mu(z, \omega) = M_\mu(z)$ a.s. in $\omega$. To do this it suffices to show that for every $\omega \in \Omega_4$ and $z_1, z \in \mathbb{R}^n$, $M_\mu(z, \tau_{z_1} \omega) = M_\mu(z, \omega)$. We can deduce
\begin{align*}
M_\mu(z, \tau_{z_1} \omega)
&= \lim_{t \rightarrow \infty} t^{-1} m_\mu(tz+z_1, z_1, \omega) \\
&= \lim_{t \rightarrow \infty} t^{-1} (m_\mu(tz+z_1, 0, \omega) + C|z_1|) \\
&= \lim_{t \rightarrow \infty} t^{-1} (m_\mu(tz, 0, \omega) + 2C|z_1|) \\
&= \lim_{t \rightarrow \infty} t^{-1} m_\mu(tz, 0, \omega) = M_\mu(z, \omega),
\end{align*}
where we have used Proposition \ref{jointlystationary}, (\ref{metricproblemoscbound}), Lemma \ref{lemma1}, and the ergodic hypothesis. Some other properties of $M_\mu$ follow readily. It is clear that $M_\mu$ is positively 1-homogeneous, and then (\ref{boundarycondition}) implies that $M_\mu$ is nonpositive. Applying (\ref{finaloscbound}) and (\ref{metricproblemoscbound}) shows that $M_\mu(z)$ is Lipschitz continuous in $z$.

To finish the proof, we show that $M_\mu$ is the solution of (\ref{effhamiltonian}) via a perturbed test function method. Suppose that $\varphi$ is a smooth function, $z_0 \neq 0$, and $z \mapsto M_\mu(z) - \varphi(z)$ has a strict global maximum at $z = z_0$. We would like to show that $\overline{H}(p+ D\varphi(z_0)) \leq \mu.$ Arguing by contradiction, we assume that $\theta:= \overline{H}(p+ D\varphi(z_0)) - \mu > 0.$ Let $\lambda_j$ be the subsequence from Remark \ref{convsubsequence} that satisfies (\ref{convsubsequence1}). Set $p_1 := p + D\varphi(z_0)$, and define the perturbed test function $\phi_j(z) := \varphi(z) + \left(\lambda_j v^{\lambda_j}\left(\frac{z}{\lambda_j}, \omega; p_1\right) + \overline{H}(p_1)\right).$ We claim that for sufficiently large $j$ and sufficiently small $r > 0$, $\phi_j$ satisfies
\begin{equation} \label{whatwewant}
\bar{J}-\int \lambda_j^{n} J_{\lambda_j}(y) \exp(-y\cdot p) \exp(\lambda_j^{-1}(\varphi_j(z - y, \omega) - \varphi_j(z, \omega)))dy - c\left(\frac{z}{\lambda_j}\right) \geq \mu+\frac{\theta}{2}
\end{equation}
in the viscosity sense for $z \in B(z_0, r)$. To prove this claim, select a smooth function $\psi$ and a point $z_1 \in B(z_0, r)$ at which $\varphi_j - \psi$ has a global minimum. It then follows that $z \mapsto v^{\lambda_j}(z, \omega; p_1) - \lambda_j^{-1}(\psi(\lambda_j z) - \varphi(\lambda_j z))$ has a global minimum at $z = \frac{z_1}{\lambda_j}$. Since $v^\lambda(\cdot, p_1)$ solves (\ref{approxcellproblemomega}) with $p_1$, we have that
\begin{multline} \label{ptfthing}
\lambda_j v^{\lambda_j}\left(\frac{z_1}{\lambda_j}, \omega; p_1\right) + \bar{J} - \int J(y) \exp(-y\cdot p_1) \\ \exp\left(\frac{\psi(z_1-\lambda_j y) - \psi(z_1)}{\lambda_j} + \frac{\varphi(z_1-\lambda_j y) - \varphi(z_1)}{\lambda_j}\right) dy - c\left(\frac{z_1}{\lambda_j}\right) \geq 0.
\end{multline}
We know that for $\lambda_j$ sufficiently small, $-\lambda_j v^{\lambda_j}\left(\frac{z_1}{\lambda_j}, \omega; p_1\right) \geq \overline{H}(p_1) - \frac{\theta}{8},$ so we can
make the change of variables $y \mapsto \frac{y}{\lambda_j}$ in (\ref{ptfthing}) to obtain (\ref{whatwewant}) for $\lambda_j$ and $r$ sufficiently small. Comparison (Lemma \ref{comparisonlemma}) with (\ref{whatwewant}) and (\ref{scaledmetricproblem}) now implies that
$$
\sup_{B(z_0, r)} [m_\mu^{\lambda_j}(\cdot) - \varphi_j(\cdot)] \leq \sup_{|z- z_0| > r} [ m_\mu^{\lambda_j}(z) - \varphi_j(z)]
$$
Sending $j \rightarrow \infty$, we obtain a contradiction to the assumption that $M_\mu(z) - \varphi(z)$ has a strict global maximum at $z_0$. The proof that $M_\mu$ is a supersolution of (\ref{effhamiltonian}) is similar, and thus we can conclude that $M_\mu$ is a solution of (\ref{effhamiltonian}). The well-posedness of (\ref{effhamiltonian}) then implies that $M_\mu = \overline{m}_\mu$.
\end{proof}
To conclude this section, we give a characterization of $\overline{m}_\mu$ that will be used in the homogenization proof of Section \ref{sec:homogproof}.
\begin{lemma}
For each $\mu < \overline{H}(p)$,
\begin{equation} \label{randomlemma2}
\overline{m}_\mu(z; p) = \inf_{q \in \mathbb{R}^n} \{z \cdot q: \overline{H}(p+q) \geq \mu\}.
\end{equation}
\end{lemma}
\begin{proof}
Since the right hand side of (\ref{randomlemma2}) is a supersolution of (\ref{effhamiltonian}), the ``$\leq$'' inequality in (\ref{randomlemma2}) follows from Proposition \ref{comparison}.

It remains to prove the ``$\geq$'' inequality. First note that $\overline{m}_\mu$ is positively 1-homogeneous in $z$, i.e. for every $t \geq 0$, $\overline{m}_\mu(tz; p) = t\overline{m}_\mu(z; p).$ This follows from the fact that \eqref{effhamiltonian} is scale-invariant and that $\overline{m}_\mu$ is unique.
We next show that $\overline{m}_\mu(z; p)$ is a concave function of $z$. This is done by using the superadditivity property \eqref{superadditivity}. We have for any $z_1 \in \mathbb{R}^n$,
\begin{align*}
\overline{m}_\mu(z; p) &= \lim_{t \rightarrow \infty} t^{-1} \hat{m}_\mu(tz, 0, \omega) \\
&\geq \lim_{t \rightarrow \infty} t^{-1}(\hat{m}_\mu(tz_1, 0, \omega) + \hat{m}_\mu(tz, tz_1, \omega)) \\
&= \overline{m}_\mu(z_1;p) + \overline{m}_\mu(z - z_1; p),
\end{align*}
where the last equality can be proven
from \eqref{metricproblemconvergence} by using the ergodic theorem and Egorov's Theorem (see \cite{takispaper2}). This implies that $\overline{m}_\mu$ is a concave function. If (\ref{randomlemma2}) fails, then we can find a point $z \neq 0$ such that $\overline{m}_\mu(z; p)$ is differentiable at $z$ and $\overline{m}_\mu(z; p) < \inf\{z \cdot q: \overline{H}(p+q)\geq\mu\},$ because a concave function is differentiable almost everywhere. Since $\overline{m}_\mu$ is 1-homogeneous, $\overline{m}_\mu(z; p) = z \cdot D\overline{m}_\mu(z; p)$, so it follows that $\overline{H}(p+D\overline{m}_\mu(z; p)) < \mu.$ This contradicts the elementary viscosity-theoretic fact that $\overline{m}_\mu$ must satisfy $\overline{H}(p+D\overline{m}_\mu) = \mu$ at any point of differentiability.
\end{proof}
We will next strengthen the lemma and show that the optimal value of $q$ in \eqref{randomlemma2} must lie on the boundary of $\{q: \overline{H}(p+q) \geq \mu\}$.
\begin{lemma} \label{randomlemma}
Fix $p \in \mathbb{R}^n$ and $\mu < \overline{H}(p)$. We have that
\begin{equation} \label{randomlemma3}
\overline{m}_\mu(z; p) = \inf_{q \in \mathbb{R}^n} \{z \cdot q: \overline{H}(p+q) = \mu\}.
\end{equation}
In addition, suppose that $q_0 \in \mathbb{R}^n$ satisfies $\overline{H}(p + q_0) = \mu$, then there exists $z_0 \neq 0$ such that $\overline{m}_\mu(z_0) = z_0 \cdot q_0$.
\end{lemma}
\begin{proof}
To show \eqref{randomlemma3}, note that it follows from \eqref{randomlemma2}, and the fact that if the optimal $q$ were in the interior of $\{q:\overline{H}(p+q) \geq \mu\}$, then there would exist a ball $B(q, \epsilon) \subset \{\overline{H}(p+q) \geq \mu\}$, and then we could take $q-\epsilon y$, and $y\cdot (q + \epsilon y) = y \cdot q - \epsilon |y|^2$, so $q$ is not the optimal value, a contradiction.

Now, let $q_0$ satisfy $H(p + q_0) = \mu$. Define $S_\gamma = \{z \in \mathbb{R}^d: z\cdot q_0 \leq \gamma\},$ and define
$$
\overline{\gamma} := \inf\{\gamma \geq 0: S_\gamma \cap \{\overline{H}(p + \cdot) \geq \mu\} = \emptyset\}.
$$
We note that $\gamma > 0$, because 0 is in the interior of the superlevel set $\{\overline{H}(p + \cdot) \geq \mu\}$. In addition, $S_{\overline{\gamma}} \cap \{\overline{H}(p + \cdot) \geq \mu\} \neq \emptyset.$ This is true because if this intersection were empty, then the distance between $S_{\overline{\gamma}}$ and $\{\overline{H}(p + \cdot) \geq \mu\}$ is positive, due to the fact that $\{\overline{H}(p + \cdot) \geq \mu\}$ is a compact set because $\overline{H}$ is negatively coercive. In this case, $\overline{\gamma}$ would not be optimal because then the plane can be adjusted further so that $S_{\overline{\gamma} + \epsilon} \cap \{\overline{H}(p + \cdot) \geq \mu\} = \emptyset$ for small $\epsilon$.

Choose $x_0 \in S_{\overline{\gamma}} \cap \{\overline{H}(p + \cdot) \geq \mu\}$. We show that for each $q \in \{\overline{H}(p + \cdot) \geq \mu\}$, $q \cdot x_0 \geq \overline{\gamma}$. This is true because for all $\gamma > \overline{\gamma}$, $q \cdot x_0 > \gamma$ due to the fact that $S_{\gamma} \cap \{\overline{H}(p + \cdot) \geq \mu\} = \emptyset$, which allows us to take $\gamma \rightarrow \overline{\gamma}$ to conclude that $q \cdot x_0 \geq \overline{\gamma}$, and so by \eqref{randomlemma2}, we know that $\overline{m}_\mu(x_0; p) \geq \overline{\gamma}.$ However, we have by definition that $x_0 \cdot q_0 \leq \overline{\gamma}$, which allows us to conclude that $q_0 \cdot x_0 = \overline{\gamma}$, and \eqref{randomlemma2} allows us to conclude that $\overline{m}_\mu(x_0; p) = \overline{\gamma} = x_0 \cdot q_0,$ which finishes the proof.
\end{proof}
\section{Proof of Homogenization and Theorem \ref{maintheorem}}
\label{sec:homogproof}
In this section we will finish the proof of our main result, Theorem \ref{maintheorem}. The first subsection covers the key step which is to improve the result of Proposition \ref{firstconvergenceprop} and show that $\lambda v^\lambda$ converges uniformly to $\overline{H}(p)$ on balls of radius $\sim \lambda^{-1}$ almost surely, i.e.
\begin{equation} \label{upgradedconvergence}
\lim_{\lambda \rightarrow 0} \sup_{B_{\frac{R}{\lambda}}} | \lambda v^\lambda(\cdot, \omega; p) + \overline{H}(p)| = 0.
\end{equation}
In the second subsection, \eqref{upgradedconvergence} will be used to apply a perturbed function method to show that the $\phi^\epsilon$ from Section \ref{sec:maintheorem}, which satisfy the equation
\begin{equation} \label{phiequation}
\left\{
\begin{array}{ll}
\displaystyle \phi^\epsilon_t + \bar{J} - \int J(y) \exp\left(\frac{\phi^\epsilon(x - \epsilon y) - \phi^\epsilon(x)}{\epsilon}\right) dy - \frac{f\left(\frac{x}{\epsilon}, \exp(\epsilon^{-1} \phi^\epsilon)\right)}{\exp(\epsilon^{-1}\phi^\epsilon)} = 0 & \text{ on } \mathbb{R}^n \times (0, \infty) \\
\displaystyle \phi^\epsilon = \epsilon \log(u_0) & \text{ on } \mathrm{int}(G_0) \times \{0\} \\
\displaystyle \phi^\epsilon(x, t) \rightarrow -\infty \text{ as } t \rightarrow 0 & \text{ for } x \in \mathbb{R}^n \backslash \overline{G_0},
\end{array} \right.
\end{equation}
converge locally uniformly to $\phi$, the solution of the homogenized equation \eqref{effectiveequation}. In other words, we prove the homogenization of \eqref{phiequation}, and this result implies Theorem \ref{maintheorem}.
\subsection{Almost Sure Convergence of $v^\lambda$}
We first need to use the methods of Section \ref{sec:metricproblem} to develop a characterization of the maximum of $p \mapsto \overline{H}(p)$, which is guaranteed to be attained by \eqref{hbarcoercive}. This characterization involves the solvability of the equation
\begin{equation} \label{Hcharacterization}
\bar{J}-\int J(y) \exp(-y \cdot p + v(z-y, \omega) - v(z, \omega))dy - c(z, \omega) \geq \mu.
\end{equation}
over the space $\mathrm{Lip}(\mathbb{R}^n)$, which denotes Lipschitz continuous functions on $\mathbb{R}^n$. This observation will play an important role in the proof of Theorem \ref{asconvergence}.
\begin{lemma} \label{lemma2}
Suppose that $p \in \mathbb{R}^n$ satisfies $\displaystyle \overline{H}(p) = \max_{q \in \mathbb{R}^n} \overline{H}(q).$ Then there exists a set of full probability $\Omega_5 \subset \Omega_4$ such that the following formula holds for $\omega \in \Omega_5$:
\begin{equation} \label{lemma1equation}
\overline{H}(p) = \sup\{\mu: \text{ there exists } v \in \mathrm{Lip}(\mathbb{R}^n) \text{ satisfying } (\ref{Hcharacterization})\},
\end{equation}
\end{lemma}
\begin{proof}
Let $\tilde{H}(\omega)$ be defined to be the right hand side of (\ref{lemma1equation}). The stationarity of the coefficients of \eqref{Hcharacterization}, the assumption that $v$ is Lipschitz continuous, and the ergodicity assumption imply that $\tilde{H}(\omega) = \tilde{H}$ on $\Omega_5\subset \Omega_4$, where $\mathbb{P}[\Omega_5] = 1$. In addition, the existence of $w_\nu$ for any $\nu > 0$ implies that $\tilde{H} \geq \overline{H}(p) - \nu$ for $\nu$ arbitrarily small. This means that $\tilde{H} \geq \overline{H}(p)$. It remains to show the reverse inequality, and this follows in a similar fashion as in \cite{ben}. Select $\mu < \tilde{H}$. Then there exists $v \in \mathrm{Lip}(\mathbb{R}^n)$ satisfying (\ref{Hcharacterization}). Define $\tilde{v}(z) = v(z) - v(0)$ and $\tilde{u}(z) := -a|z|-C.$ Then upon taking $a$ and $C$ sufficiently large $\tilde{u}(z)$ is a subsolution of (\ref{Hcharacterization}) on $\mathbb{R}^n \backslash B(0, 1)$ satisfying $\tilde{u} \leq \tilde{v}$ on $\overline{B(0, 1)}$. Define
\begin{align*}
\mathcal{S}(\omega) &:= \{h \in \mathrm{Lip}(\mathbb{R}^n): h \geq \tilde{u} \text{ on } B(0, 1), h \text{ satisfies } (\ref{Hcharacterization})\}, \\
 s_\mu(z, z_1, \omega) &:= \inf_{h \in \mathcal{S}(\omega)} [h(z) - h(z_1)].
\end{align*}
Because $\tilde{v} \in \mathcal{S}(\omega)$, we can apply Perron's method, and we see that $s_\mu$ is a solution of
\begin{equation}
\left\{
\begin{array}{ll}
\displaystyle \bar{J}-\int J(y) \exp(-y \cdot p + s_\mu(z-y, \omega) - s_\mu(z, \omega))dy - c(z, \omega) = \mu & \text{ on } \mathbb{R}^n \backslash \overline{B(0, 1)}, \\
s_\mu = \tilde{u} & \text{ on } \overline{B(0, 1)}, \end{array}
\right.
\end{equation}
satisfying $\tilde{u} \leq s_\mu \leq \tilde{v}$ on $\mathbb{R}^n \backslash B(0, 1)$. Note that similar to the situation in Proposition \ref{metricproblemwellposed}, $s_\mu$ is piecewise Lipschitz continuous. The function $s_\mu$ is superadditive by construction, and we can check that it is jointly stationary using the fact that a function $w(z) \in \mathcal{S}(\tau_{z_1} \omega)$ if and only if $w(y+z) \in \mathcal{S}(\omega)$, because the coefficients of \eqref{Hcharacterization} are stationary.
Now if we define $s^\epsilon_\mu(z) := \epsilon s_\mu\left(\frac{z}{\epsilon}\right)$, then $s^\epsilon_\mu$ satisfies the rescaled system
$$
\left\{
\begin{array}{ll}
\displaystyle \bar{J}-\int \epsilon^{n} J_\epsilon(y) \exp\left(-y \cdot p + \frac{s^\epsilon_\mu(z - y, \omega) - s^\epsilon_\mu(z, \omega)}{\epsilon}\right)dy - c\left(\frac{z}{\epsilon}, \omega\right) = \mu & \text{ on } \mathbb{R}^n \backslash \overline{B(0, \epsilon)}, \\
s^\epsilon_\mu = \tilde{u}^\epsilon & \text{ on } \overline{B(0, \epsilon)}, \end{array}
\right.
$$
with $\tilde{u}^\epsilon(z) = \epsilon \tilde{u}(\frac{z}{\epsilon})$. We remark that the condition $\mu < \overline{H}(p)$ was only used in Proposition \ref{metricproblemhomogthm} in order to ensure the existence of the $m^\epsilon_\mu$, and we have the existence of $s^\epsilon_\mu$ \emph{a priori} in this case. (\ref{metricproblemoscbound}) holds for $s^\epsilon_\mu$ by the same argument as in Proposition \ref{metricproblemoscprop}, and we have established superadditivity and joint stationarity for $s^\epsilon_\mu$ above. Therefore, we can conclude that using the subadditive ergodic theorem as in Proposition \ref{metricproblemhomogthm} that $s^\epsilon_\mu(z) \rightarrow \overline{s}_\mu(z)$ a.s. in $\omega$, where $\overline{s}_\mu$ satisfies the equation $\overline{H}(p+ D\overline{s}) = \mu \text{ on } \mathbb{R}^n \backslash \{0\}.$ This means that $\overline{H}(p) \geq \mu$, and since $\mu$ was an arbitrary constant satisfying $\mu < \tilde{H}$, this implies $\overline{H}(p) \geq \tilde{H}$.
\end{proof}
We will now prove \eqref{upgradedconvergence}.
\begin{theorem} \label{asconvergence}
There exists a set $\Omega_6 \subset \Omega_5$ of full probability such that for every $R > 0$, $p \in \mathbb{R}^n$, \eqref{upgradedconvergence} holds.
\end{theorem}
\begin{proof}
Lemma \ref{plipschitzlemma} implies that we can reduce to the case of a fixed $p \in \mathbb{R}^n$; $\Omega_6$ is then obtained by intersecting the full probability subsets $\Omega_p$ created for rational $p$. We follow the methods of \cite{takispaper} and \cite{takispaper2} and separate the proof into parts. The first part shows that it suffices to consider the convergence of $\lambda v^\lambda(0, \omega)$. The next part considers this convergence, and we show that $\limsup_{\lambda \rightarrow 0} \lambda v^\lambda(0, \omega) = \liminf_{\lambda \rightarrow 0} \lambda v^\lambda(0, \omega)$ by considering two cases, the case where $\overline{H}(p) = \max{\overline{H}}$ and the case where $\overline{H}(p) < \max{\overline{H}}$.

\emph{Step 1: Reducing the question to convergence at the origin.} Proposition \ref{firstconvergenceprop} and Remark \ref{liminfremark} imply that for $\omega \in \Omega_5$,
\begin{equation} \label{HbarHhat}
\limsup_{\lambda \rightarrow 0} \lambda v^\lambda(0, \omega) = -\overline{H}(p) \geq -\hat{H}(p):= \liminf_{\lambda \rightarrow 0} \lambda v^\lambda(0, \omega).
\end{equation}
In this part we claim that for each $R > 0$,
\begin{equation} \label{step1}
\limsup_{\lambda \rightarrow 0}[\sup_{B_{\frac{R}{\lambda}}}\lambda v^\lambda(\cdot, \omega)] = -\overline{H}(p) \text{ and } \liminf_{\lambda \rightarrow 0} [\inf_{B_{\frac{R}{\lambda}}}\lambda v^\lambda(\cdot, \omega)] = -\hat{H}(p) \text{ a.s. in } \omega.
\end{equation}
We will prove the first identity of \eqref{step1}, the other one following in a similar fashion, and both proceeding in a similar fashion as the proof of Proposition 7.1 of \cite{takispaper}. By (\ref{HbarHhat}) and Egorov's theorem, for every $\gamma > 0$, there exists $\bar{\lambda}(\gamma) > 0$ and a set $E_\gamma$ such that $\mathbb{P}[E_\gamma] \geq 1-\gamma$ and for every $0 < \lambda < \bar{\lambda}(\gamma)$, $\sup_{\omega \in E_\gamma} \lambda v^\lambda(0, \omega) + \overline{H} \leq \gamma.$
The Ergodic Theorem gives us for each $\gamma > 0$ a full probability set $F_\gamma$ such that for every $\omega \in F_\gamma$,
$$
\lim_{R \rightarrow \infty} \dashint_{B_R} \mathbbm{1}_{E_\gamma}(\tau_y\omega) dy = \mathbb{P}[E_\gamma] \geq 1-\gamma.
$$
Now define $F_0:= \bigcap_{j = 1}^\infty F_{2^{-j}}.$ Then $\mathbb{P}(F_0) = 1$. Fix $\omega \in F_0$ and $R, \gamma > 0$, with $\gamma = 2^{-j}$ for some natural number $j$. We know that for $\lambda$ sufficiently small,
\begin{equation} \label{measureofset}
|\{y \in B_\frac{R}{\lambda}: \tau_y \omega \in E_\gamma\}| \geq (1-2\gamma) |B_{\frac{R}{\lambda}}|.
\end{equation}
Using (\ref{finaloscbound}), we have for $r \in (\gamma R, R)$ and $\lambda$ sufficiently small,
\begin{equation} \label{blah5}
\sup_{z \in B_R} \osc_{B(\frac{z}{\lambda}, \frac{r}{\lambda})} \lambda v^\lambda(\cdot, \omega) \leq Cr.
\end{equation}
Select any $z \in B_{\frac{R}{\lambda}}$. We claim that due to \eqref{measureofset} we find a point $z_1 \in B_{\frac{R}{\lambda}}$ with $|z - z_1| \leq C\lambda^{-1}\gamma^{\frac{1}{2n}} R \text{ and } \tau_{z_1} \omega \in E_\gamma.$ This is true because the ball $B(z,C\lambda^{-1}\gamma^{\frac{1}{2n}} R)$ is too large to lie in the complement of $\{z_1 \in \mathbb{R}^d: \tau_{z_1} \omega \in E_\gamma\}$. Now due to (\ref{blah5}) and the fact that $v^\lambda$ are stationary, we can deduce that for each $\lambda$ sufficiently small,
\begin{align*}
\lambda v^\lambda(z, \omega) + \overline{H} &\leq |\lambda v^\lambda(z, \omega) - \lambda v^\lambda(z_1, \omega)| + \lambda v^\lambda(z_1, \omega) + \overline{H} \\
&\leq C(C\gamma^{\frac{1}{2n}} R) + \lambda v^\lambda(0, \tau_y\omega) + \overline{H} \leq C\gamma^{\frac{1}{2n}} R + \gamma,
\end{align*}
and hence for each $\omega \in F_0$ and $R > 0$, $\limsup_{\lambda \rightarrow 0} \sup_{z \in B_{\frac{R}{\lambda}}}(\lambda v^\lambda(z, \omega) + \overline{H}) \leq 0,$ because $\gamma$ can be taken to be arbitrarily small. This implies (\ref{step1}) for $\omega \in F_0 \cap \Omega_5$.

\emph{Step 2: The convergence at the origin in the case $\overline{H}(p) = \max{\overline{H}}$.} Now that we have reduced to considering convergence of the origin, we will consider two cases. First we consider the case when $p$ is the vector in $\mathbb{R}^n$ where the maximum of $p \mapsto \overline{H}(p)$ is attained. We'd like to show that $\overline{H}(p) = \hat{H}(p)$. It suffices to show that $\hat{H}(p) \leq \tilde{H}$, where $\tilde{H}$ is the right hand side of (\ref{lemma1equation}) from Lemma \ref{lemma2}. If $\hat{H}(p) \leq \tilde{H}$, then our assumption and Lemma \ref{lemma2} show that $\overline{H}(p) = \max_{\mathbb{R}^n} \overline{H} = \tilde{H},$ so we can conclude from \eqref{HbarHhat} that $\overline{H}(p) = \hat{H}(p)$. To show that $\hat{H}(p) \leq \tilde{H}$, note that for $\omega \in \Omega_5$, \eqref{hhatfullprob} holds. Therefore, we can use the same methods as in Section \ref{sec:effhamiltonian}. Consider $w^\lambda(z, \omega) := v^\lambda(z, \omega) - v^\lambda(0, \omega),$ and as before, let $w^\lambda_\theta = w^\lambda \ast \rho_\theta$ be the mollification of $w^\lambda$. Then we can pass to the limit as $\lambda \rightarrow 0$ along a subsequence to find a Lipschitz continuous solution $w$ of
$$
\bar{J}-\int J(y) \exp(-y\cdot p + w(z-y, \omega) - w(z, \omega))dy - c(z, \omega) \geq \hat{H}(p) - \nu \text{ in } \mathbb{R}^n
$$
for any $\nu > 0$ arbitrarily small. Therefore, we have by definition of $\tilde{H}$ that $\hat{H}(p) \leq \tilde{H}$.

\emph{Step 3: The convergence at the origin in the case $\overline{H}(p) < \max{\overline{H}}$.}
Now we consider the case where $p$ is not a vector where the maximum of $\overline{H}$ is attained. Without loss of generality, we can assume that $\max_{q \in \mathbb{R}^n} \overline{H}(q) = \overline{H}(0) > \overline{H}(p).$
We argue by contradiction and suppose that $\xi := \hat{H}(p) - \overline{H}(p) > 0.$ Select a subsequence $\lambda_j$, possibly depending on $\omega$, such that $\lim_{j \rightarrow \infty} \lambda_j v^{\lambda_j}(0, \omega; p)= -\hat{H}(p)$ Define $\mu := \overline{H}(p)$. According to Lemma \ref{randomlemma}, we can pick $z_0 \neq 0$ such that $\overline{m}_\mu(z_0; 0) = z_0 \cdot p.$
Let $\eta$ be a small parameter to be selected below, and define $\varphi_j(z) := (z + z_0) \cdot p + \lambda_j v^{\lambda_j}\left(\frac{z}{\lambda_j}, \omega; p\right) + \eta |z|^2$. We can easily see that for $r, \eta$ sufficiently small and $j$ large enough, we have that $\varphi_j$ satisfies
$$
\bar{J}-\int \lambda_j^{n} J_{\lambda_j}(y) \exp(\lambda_j^{-1}(\varphi_j(z - y, \omega) - \varphi_j(z, \omega)))dy - c\left(\frac{z}{\lambda_j}\right) \geq \mu + \frac{\xi}{2} \text{ in } B(0, r).
$$
Now denote $m^\epsilon_\mu$ to be the solution of (\ref{scaledmetricproblem}) with $p = 0$, which exists because $\mu < \overline{H}(0)$ by assumption. Because $\varphi_j$ is a strict supersolution of (\ref{scaledmetricproblem}), our comparison result Lemma \ref{comparisonlemma} implies that
\begin{equation} \label{contradiction}
\min_{z \in B(0, r)} (\varphi_j(z) - m^{\lambda_j}_\mu(z, -z_0, \omega; 0)) = \min_{z \in E} (\varphi_j(z) - m^{\lambda_j}_\mu(z, -z_0, \omega; 0)).
\end{equation}
Here $E$ is the same as it was in Lemma \ref{comparisonlemma}. We will get a contradiction by taking $j \rightarrow \infty$ on the right hand side of \eqref{contradiction}. Group the terms as follows:
\begin{multline*}
\varphi_j(x) - m^{\lambda_j}_\mu(z, -z_0, \omega; 0) = \left(\lambda_j v^{\lambda_j}\left(\frac{z}{\lambda_j}, \omega; p\right) + \eta |z|^2\right) + (p \cdot (z + z_0) - \overline{m}_\mu(z+z_0; 0)) \\
+ (\overline{m}_\mu(z + z_0;0) - m^{\lambda_j}_\mu(z, -z_0, \omega; 0)).
\end{multline*}
The last term converges locally uniformly to 0 by the homogenization of the metric problem, the second term is nonnegative and vanishes at $z = 0$, and the first term satisfies, by (\ref{step1}) and the fact that $E$ is bounded,
$$
\lim_{j \rightarrow 0} \inf_{z \in E} \left(\lambda_j v^{\lambda_j}\left(\frac{z}{\lambda_j}, \omega; p\right) + \eta |z|^2\right) \geq -\hat{H}(p) + \eta (r+C)^2 > -\hat{H}(p) = \lim_{j \rightarrow \infty} (\lambda_j v^{\lambda_j}(0, \omega; p)).
$$
This means that \eqref{contradiction} is impossible for $j$ large enough. Therefore, $\overline{H}(p) = \hat{H}(p)$.
\end{proof}
\subsection{Homogenization Result, Proof of Theorem \ref{maintheorem}}
We will finally finish the proof of our main result, Theorem \ref{maintheorem}.
\begin{proof}[Proof of Theorem \ref{maintheorem}] \label{mainhomogtheorem} \let\qed\relax
To prove our main theorem, the primary remaining step is the following theorem demonstrating the homogenization of \eqref{phiequation}.
\begin{theorem}
For $\omega \in \Omega_6$, $\phi^\epsilon$ converges locally uniformly to $\phi$ on $\mathbb{R}^n \times (0, \infty)$.
\end{theorem}
Theorem \ref{mainhomogtheorem} in conjunction with \eqref{hopfcole} shows that $u^\epsilon \rightarrow 0$ on $\{\phi < 0\}$. The argument to show that $u^\epsilon \rightarrow 1$ on $\mathrm{int}\{\phi = 0\}$ is done in exactly the same manner as in Section 5 of \cite{almostperiodicpaper}. Therefore, to conclude the proof of Theorem \ref{maintheorem} it suffices to prove Theorem \ref{mainhomogtheorem}, and this follows via a perturbed test function method, similar to the proof of Theorem 4.1 of \cite{almostperiodicpaper}.
\end{proof}
\begin{proof}[Proof of Theorem \ref{mainhomogtheorem}]
For each $(x, t) \in \mathbb{R}^n \times (0, \infty)$, we define
\begin{equation} \label{phistar}
\phi^*(x, t) = \limsup_{\epsilon \rightarrow 0, (x', s) \rightarrow (x, t)} \phi^\epsilon(x', s), \phi_*(x, t) = \liminf_{\epsilon \rightarrow 0, (x', s) \rightarrow (x, t)} \phi^\epsilon(x', s)
\end{equation}
to be the half-relaxed upper and lower limits (see \cite{usersguide}); note that the local uniform bounds on $\phi^\epsilon$ from Lemma 4.5 of \cite{almostperiodicpaper} implies that $\phi^*(x, t), \phi_*(x, t) \in \mathbb{R}$ for all $(x, t) \in \mathbb{R}^n \times (0, \infty)$. We show that $\phi^*$ is a subsolution of
\begin{equation} \label{variationalinequality}
\max(\phi^*_t + \overline{H}(D\phi^*), \phi^*) \leq 0 \text{ in } \mathbb{R}^n \times (0, \infty);
\end{equation}
the proof that $\phi_*$ is a supersolution of \eqref{variationalinequality} follows similarly with some changes in dealing with the $(u^\epsilon)^{-1} f(u^\epsilon)$ term, as noted in \cite{almostperiodicpaper} and \cite{majda}. Take a smooth test function $\varphi$ and a point $(x_0, t_0)$ such that $(x, t) \mapsto \phi^*(x, t) - \varphi(x, t)$
has a strict global maximum at $(x_0, t_0)$ (with $t_0 > 0$). Because $\phi^\epsilon \leq 0$ by (\ref{hopfcole}) and comparison for \eqref{phiequation}, showing (\ref{variationalinequality}) reduces to showing that $\phi^*_t(x_0, t_0) + \overline{H}(D\phi^*(x_0, t_0)) \leq 0$.

Assume for a contradiction that $\varphi_t(x_0, t_0) + \overline{H}(D\varphi(x_0, t_0)) = \theta > 0.$ Set $p_0 := D\varphi(x_0, t_0)$ and for $\epsilon > 0$ define the perturbed test function $\varphi^\epsilon(x, t) := \varphi(x, t) + \epsilon v^\epsilon\left(\frac{x}{\epsilon}, \omega; p_0\right).$ $v^\epsilon$ is the solution to the approximated cell problem (\ref{approxcellproblemomega}) with $\lambda = \epsilon$ and $p = p_0$. We claim that for $r, \epsilon$ sufficiently small,
\begin{equation} \label{periodichomogthing}
\varphi^\epsilon_t + \bar{J} - \int J(y) \exp\left(\frac{\varphi^\epsilon(x - \epsilon y) - \varphi^\epsilon(x)}{\epsilon}\right) dy - c\left(\frac{x}{\epsilon}\right) \geq \frac{\theta}{2} \text{ in } B(x_0, r) \times (t_0 - r, t_0 + r)
\end{equation}
holds in the viscosity sense. To show (\ref{periodichomogthing}), select another smooth test function $\psi$ and a point $(x_1, t_1) \in B(x_0, r) \times (t_0 - r, t_0 + r)$ such that $(x, t) \mapsto (\varphi^\epsilon - \psi)(x, t)$ has a global minimum at $(x_1, t_1)$. This means that $(z, t) \mapsto v^\epsilon(z, \omega; p_0) - \frac{1}{\epsilon}(\psi(\epsilon z, t) - \varphi(\epsilon z, t))$ has a global minimum at $(\frac{x_1}{\epsilon}, t_1).$
In particular this implies that $\varphi_t(x_1, t_1) = \psi_t(x_1, t_1)$, because $v^\epsilon$ doesn't depend on $t$. We know that $v^\epsilon$ solves (\ref{approxcellproblemomega}) with $p_0$, so we have that
\begin{multline*}
\epsilon v^{\epsilon}\left(\frac{x_1}{\epsilon}, \omega; p_0\right) + \bar{J} - \int J(y) \exp(-y\cdot p_0) \\ \exp\left(\frac{\psi(x_1-\epsilon y) - \psi(x_1)}{\epsilon} + \frac{\phi(x_1-\epsilon y) - \phi(x_1)}{\epsilon}\right) dy - c\left(\frac{x_1}{\epsilon}\right) \geq 0.
\end{multline*}
We know by Theorem \ref{asconvergence} that for $\epsilon$ sufficiently small, $-\epsilon v^\epsilon\left(\frac{x_1}{\epsilon}\right) \geq \overline{H}(p_0) - \frac{\theta}{8} = -\varphi_t\left(\frac{x_1}{\epsilon}, t_1\right) + \frac{7\theta}{8},$ and because $\phi$ is smooth, as $\epsilon \rightarrow 0$ we have $\epsilon^{-1} (\phi(x_1-\epsilon y) - \phi(x_1)) \rightarrow -D\phi(x_1).$ Because $(x_1, t_1)$ is close to $(x_0, t_0)$ when $r$ is small, taking $\epsilon$ and $r$ sufficiently small yields that (\ref{periodichomogthing}) holds in the viscosity sense.
In addition, we can use (\ref{kppcondition}) as before to see that $\phi^\epsilon$ satisfies
\begin{equation} \label{phiepeq}
\phi^\epsilon_t + \bar{J} - \int J(y) \exp\left(\frac{\phi^\epsilon(x_1 - \epsilon y) - \phi^\epsilon(x_1)}{\epsilon}\right)dy - c\left(\frac{x_1}{\epsilon}\right) \leq 0.
\end{equation}
Then given (\ref{periodichomogthing}) and (\ref{phiepeq}), we can use comparison to conclude that
$$
\max_{\overline{B(x_0, r) \times (t_0 - r, t_0 + r)}} \phi^\epsilon - \varphi^\epsilon = \max_{D} \phi^\epsilon - \varphi^\epsilon,
$$
where $D = \{B(x_0, r)^c \times [t_0 - r, t_0+r]\} \cup \{B(x_0, r) \times \{t = t_0 - r\}\}$. Upon taking $\epsilon \rightarrow 0$ this contradicts our initial assumption of $\phi^* - \varphi$ having a strict global maximum at $(x_0, t_0)$. We can show that the initial condition holds using the same argument as in \cite{almostperiodicpaper} and \cite{majda}, so this means that $\phi^*$ is a subsolution of (\ref{effectiveequation}). By comparison for (\ref{effectiveequation}), which follows due to \cite{cls}, $\phi^\ast = \phi_\ast = \phi$, which means that $\phi^\epsilon$ converges locally uniformly to $\phi$.
\end{proof}
\section*{Acknowledgements}
I'd like to thank Panagiotis Souganidis for suggesting this problem to me and to thank Panagiotis Souganidis and Benjamin Fehrman for many useful conversations. In particular I'd like to thank Panagiotis Souganidis for mentioning to me the mollification approach used in the proof of Proposition \ref{firstconvergenceprop}. In addition I'd like to thank Panagiotis Souganidis and Luis Silvestre for their numerous suggestions and advice throughout the process of writing and editing this paper.

This material is based upon work supported by the National Science Foundation Graduate Research Fellowship under Grant No. DGE-1144082.
\bibliography{FullPaper}
\bibliographystyle{abbrv}
\end{document}